\documentclass[11pt]{article}
\usepackage{amsmath}
\usepackage{amsmath,amssymb,amsthm,graphicx,mathrsfs,color}
\usepackage[numbers,sort&compress]{natbib}
\numberwithin{equation}{section}
\newcommand{\beq}{\begin{equation}}
	\newcommand{\enq}{\end{equation}}

\newtheorem{Theorem}{Theorem}[section]
\newtheorem{Lemma}[Theorem]{Lemma}
\newtheorem{Corollary}[Theorem]{Corollary}
\newtheorem{Definition}[Theorem]{Definition}

\newtheorem{Remark}[Theorem]{Remark}

\begin{document}
	\pagenumbering{arabic} \setcounter{page}{1}
	\renewcommand{\baselinestretch}{1}
	\renewcommand{\thefootnote}{\fnsymbol{footnote}}
	\allowdisplaybreaks[4]
	\linespread{1.2}
	
	\begin{center}
		{\large \bf On the attractor of the semi-dissipative Boussinesq equations revised}
		\vspace{0.5cm}\\
		{Kexin Li ${}^\dag$ and Chunyou Sun${}^{\dag,\S, *}$}\\\vspace{0.3cm}
		
		{\small ${}^\dag$ School of Mathematics and Statistics, Donghua University, Shanghai, 201620, P.R. China\\
			${}^\S$ School of Mathematics and Statistics, Lanzhou University, Lanzhou, 730000, P.R. China}\\
	\end{center}
	
	\begin{abstract}
		In this paper, we continue investigate the structure of the weak $\sigma$-attractor $\mathcal{A}$ for the 2D Boussinesq equations with viscosity and without heat diffusion. Specifically, we show that the projection of the attractor $\mathcal{A}$ onto the temperature component is  surjective, i.e., $P_{\theta}\mathcal{A}=L^2$; and, for any $r>0$, the projection of $\mathcal{A}_r$ onto the velocity component contains a infinite-dimensional Hilbert ellipsoid, which implies $P_u\mathcal{A}_r$ is infinite dimensional immediately. 
		
	\end{abstract}

	\vspace{-1 cm}
	
	\footnote[0]{\hspace*{-7.4mm}
		$^{*}$ Corresponding author.\\
		E-mail address: lkx@mail.dhu.edu.cn (K.Li), sunchy@lzu.edu.cn (C.Sun).}

	\section{Introduction}
	
	The 2D Boussinesq equations model geophysical flows such as atmospheric fronts and oceanic circulation, and play an important role in the study of Rayleigh-Bénard convection (see \cite{29, 31}). 
	
	This paper focus on the semi-dissipative case, characterized by the presence of viscosity and the absence of heat diffusion, which is known to be globally well-posed (see, e.g., \cite{dong2006, Rapha2008, adam2013, xu2025}). Leveraging this global well-posedness, we aim to investigate the long-time dynamics of the Boussinesq system, governed by the following equations:
	\begin{equation}\label{1.1}
		\begin{cases}
			\partial_t u + (u \cdot \nabla)u + \nabla p - \nu \Delta u = \theta e_2 & \text{in } \Omega \times \mathbb{R}_+, \\
			\partial_t \theta + u \cdot \nabla \theta = 0 & \text{in } \Omega \times \mathbb{R}_+, \\
			\nabla\cdot u = 0 & \text{in } \Omega \times \mathbb{R}_+, \\
			(u(x, 0),\mathbf{\theta}(x, 0)) = (u_0(x), \theta_0(x)) & \text{in } \Omega,
		\end{cases}
	\end{equation}
	where $\Omega = \mathbb{T}^2 := \mathbb{R}^2/(L\mathbb{Z}^2) \equiv [0,L]^2$ is a periodic domain in $\mathbb{R}^2$, $e_2 = (0,1)^T$, the velocity $u = (u^1, u^2)$, the pressure $p$ and the temperature $\theta$ are the unknowns. Notably, the system is dissipative in the velocity variable but non-dissipative in the temperature. In this regime, the classical theory of global attractors (such as \cite{bab1992, carvalho2013, hale1988, ladyzhenskaya1991, Temam1997} et al.) breaks down due to the lack of dissipativity in the temperature equation. 
	
	In \cite{Biswas2017}, A.Biswas, C.Foias, and A.Larios first introduced the innovative concept of the weak $\sigma$-attractor and applied for this system.
	\begin{Lemma}[\cite{Biswas2017}]\label{xiyinzidingyi}
		Let $\{S(t)\}_{t\geqslant 0}$ be the semigroup operator associated with \eqref{1.1} in $\mathcal{H}\times L^2$. Defining the weak $\sigma$-attractor \( \mathcal{A} \) of the semi-dissipative system \eqref{1.1} to be the set of all \( (u_0, \theta_0) \in \mathcal{H} \times L^2 \) with the property that:
		\begin{enumerate}
			\item[(i)] There exists a global trajectory \( ({u}(t), \theta(t)) \) defined for all \( t \in \mathbb{R} \) such that \( ({u}(t), \theta(t)) \) belongs to \(  \mathcal{H} \times L^2 \) and solves (\ref{1.1}) for all \( t \in \mathbb{R} \), and moreover \( ({u}(0), \theta(0)) = ({u}_0, \theta_0) \);
			\item[(ii)] The trajectory \( ({u}(t), \theta(t)) \) is globally bounded, i.e., the set \( \{ ({u}(t), \theta(t)) : t \in \mathbb{R} \} \) is bounded in \(  \mathcal{H} \times L^2 \).
		\end{enumerate}

		Then, \eqref{1.1} has a unique weak $\sigma$-attractor $\mathcal{A}$.
	\end{Lemma}
	
	Obviously, the introduction of weak $\sigma$-attractor solve some shortcomings of classical attractor theory, allows people can apply the attractor theory to describe the long-time behavior of semi-dissipative systems.
	
	Moreover,  they showed in \cite{Biswas2017} that the weak $\sigma$-attractor is
	not only non-trivial, but also an extremely rich proper subset of the phase space.
	In particular, the weak $\sigma$-attractor is a countable union of weakly compact
	invariant sets and attracts all bounded sets in the sense of the weak topology of  $\mathcal{H}\times L^2$, which, gave us also an applicable method to construct the weak $\sigma$-attractor.
	\begin{Lemma}[\cite{Biswas2017}]\label{yuanxitongxiyinzir}
		For every $r\geq0$, the local attractor at level $r$, denoted by $\mathcal{A}_r$, is defined as the $\omega$-limit set of the positively invariant set 
		\[
		\mathcal{B}_r := \bigl\{ ({u}_0,\theta_0) \in \mathcal{H} \times L^2 : \|\theta_0\|_{2} \leq r,\ \|{u}_0\|_{2} \leq R(r)  = \frac{2 r}{\nu \lambda_1}\bigr\}.
		\]
		That is,
		\begin{equation}\label{ar}
			\mathcal{A}_r := \omega^{wk}(\mathcal{B}_r) := \bigcap_{\tau\ge 0} \overline{\vphantom{\bigcap} \bigl\{ S(t)({u}_0,\theta_0) : t > \tau,\ ({u}_0,\theta_0)\in \mathcal{B}_r \bigr\}}^{wk},
		\end{equation}
		where the closure is taken in the weak topology of $\mathcal{H} \times L^2$.
		
		Then, 
		\begin{enumerate}
			\item[-] The relation $\displaystyle \mathcal{A} = \bigcup_{r\ge 0} \mathcal{A}_r = \bigcup_{\substack{r\ge 0,\,r\in\mathbb{Q}}} \mathcal{A}_r$ holds, where $\mathcal{A}_r$ is defined by (\ref{ar}).
			
			\item [-]$\mathcal{A}$ has empty interior in the strong (and therefore weak) topology of $\mathcal{H} \times L^2$.
			
			\item[-] $\mathcal{A}$ is a non-empty, proper subset of the phase space $\mathcal{H} \times L^2$ which moreover contains infinite-dimensional subspaces of the phase space.
			
			\item[-]The set $\mathcal{A}$ attracts all bounded sets in the weak topology of $\mathcal{H} \times L^2$.
		\end{enumerate}
		
	\end{Lemma}
	
	Subsequently, in \cite{sun2019}, the authors conducted a further analysis of the structure of this weak $\sigma$-attractor, proving that it possesses a pancake-like structure and thereby enriching the understanding of its architecture.
	\begin{Lemma}[\cite{sun2019}]\label{HS}
		For every $r\geq0$,  set\[
		\widetilde{\mathcal{B}}_r := \{(u_0,\theta_0) \in \mathcal{H}\times L^2 : \|\theta_0\|_{2} = r,\, \|u_0\|_{2} \leq R(r)= \frac{2r}{\nu\lambda_1}\},
		\]
		and define
		\begin{equation}\label{H-st}
			{H}_r := \omega^{st}(\widetilde{\mathcal{B}}_r) := \bigcap_{\tau \ge 0} \overline{\bigl\{ S(t)(u_0,\theta_0) : t > \tau,\,(u_0,\theta_0) \in \widetilde{\mathcal{B}}_r \bigr\}}^{strong},
		\end{equation}
		where the closure is taken in the strong (norm) topology of $\mathcal{H} \times L^2$.
		
		Then,
		\[
		\mathcal{A}=\bigcup_{r\geq 0} H_r.
		\] 
	\end{Lemma}
	
	However, compared with the classical attractors, weak $\sigma$-attractor pose completely different problems and difficulties in the characterization of its analytical structure, dimension theory and perturbation properties, and many results about the classical attractors are no longer valid. Therefore, while the authors of \cite{Biswas2017} gave many beautiful structure of the weak $\sigma$-attractor, they also poses some fundamental questions regarding its precise characterization, such as:
	\begin{quote}
		\textit{“if one takes the projection of \( \mathcal{A} \) onto the temperature component (say \( P_\theta({u}, \theta) := \theta \)), the following questions remain open: (a) Does \( P_\theta \mathcal{A} \) have an empty interior in \( L^2 \)? (b) Is \( P_\theta \mathcal{A} \) a proper subset of \( L^2 \)? (c) Is \( P_\theta \mathcal{A} \) dense in \( L^2 \)?”}
	\end{quote}
	\begin{quote}
		\textit{“$\cdots \cdots$, although the whole attractor is infinite dimensional, is there any kind of finite dimensionality in its constituents pieces $\mathcal{A}_r$?”}
	\end{quote}

	In this paper, we will answer these questions, our main result is that:
	\begin{Theorem}\label{main-thm}
		$ P_{\theta}\mathcal{A} =L^2$.
	\end{Theorem}
	
	Theorem \ref{main-thm} is a bit surprise: intuitively, \eqref{1.1} has some dissipativity and  $\mathcal{A}$ (or $\mathcal{A}_r$) is a kind of $\omega$-limit set, it should loss something. 
	
	With Theorem \ref{main-thm}, we can obtain further interesting  characterization about the structure of weak $\sigma$-attractor $\mathcal{A}$, for example,
	\begin{itemize}
		\item[-] Combining with the definition of the weak $\sigma$-attractor $\mathcal{A}$, Theorem \ref{main-thm} implies that: for any $\theta_0\in L^2$, there is a $u_0\in \mathcal{H}$ such that there exists a bounded complete trajectory $(u(t),\theta(t))$ of  \eqref{1.1} satisfying that $(u(0),\theta(0))=(u_0,\theta_0)$.
		
		\item[-] Combining with the  pancake-like structure $\mathcal{A} = \bigcup_{r\geq 0} H_r$ presented in Lemma \ref{HS} that given in \cite{sun2019}, we indeed can obtain that, for each $r\geq0$,
		\[
		P_{\theta}H_r=\{\theta\in L^2: \|\theta\|_2=r\}.
		\] 
		In fact, $P_{\theta}H_r\subseteq P_{\theta}\widetilde{\mathcal{B}}_r=\{\theta\in L^2: \|\theta\|_2=r\}$ is obviously; for the other side, we only need to note that, $H_r \cap H_{r'}=\varnothing$ whence $r\neq r'$ since the closure in \eqref{H-st} is the strong topology. 
		
		\item[-]  By Riesz lemma, it is easy to find $\{v_i\}_{i=1}^{\infty}\subset L^2$ satisfy $\|v_j\|_2=r$ and $\|v_i-v_j\|_2\geq \frac{r}{2}$ as $i\neq j$, then we know that  $\{u_i,v_i\}_{i=1}^{\infty}\subseteq H_r$ is not relatively compact for any $u_i\in \mathcal{H}$. This implies that it is impossible to get the asymptotical compactness of the solution semigroup
		in the strong topology of $\mathcal{H}\times L^2$ (note that, about the velocity variable $u$, the asymptotical compactness in the strong topology of $V$ was verified in \cite{sun2019}). In particular, the attraction properties about the temperature variable $\theta$ is impossible to be improved to the strong topology of $L^2$.
		
		\item[-] Obviously, from the above observation, it is impossible 
		for $\mathcal{A}_r$ to be finite fractal  dimensional in $\mathcal{H}\times L^2$ since $\mathcal{A}_r$ is non-compact.
		
		In fact, combining with the Horizontal solutions constructed in \cite{Biswas2017}, we will show that there is a infinite-dimensional Hilbert ellipsoid contained in $P_u\mathcal{A}_r$ (here $ P_u({u}, \theta) := u $ is the usual projection):
		\[
		\mathcal{E}_r^H
		:=\big\{
		\big(0,\frac{(-\partial_{x_1}^2)^{-1}\theta^H}{\nu}\big)\in\mathcal{H}:~\theta^H\in L^2_{0,H},~ \|\theta^H\|_{2}\leq r
		\big\}
		\subset P_u \mathcal{A}_r,
		\]
		where $L^2_{0,H}$ denotes the zero-mean $L^2$-functions on $\Omega$ that depend only on the horizontal variable $x_1$; which, implies that the Hausdorff and fractal dimensions of $P_u \mathcal{A}_r$ in $\mathcal{H}$, as well as the dimension of $P_u H_r$,  is infinite whence $r>0$. See the discussion in Section \ref{s5}.
	\end{itemize}

	Our argument is elementary and straightforward. In fact, we only begin with an analysis of Boussinesq system with projection, see \ref{touyingxitong} in Section \ref{timu4}, and then let $N\to\infty$ to approximate the original Boussinesq system \eqref{1.1}. By analyzing the solutions of \ref{touyingxitong}, we observe that \eqref{1.1} and \ref{touyingxitong} share identical dissipation properties, differing only in the space where the temperature variable lives. This analysis allows us to construct the weak $\sigma$-attractor $\mathcal{A}_N\subseteq \mathcal{H}\times X_N$ for \ref{touyingxitong}, and obviously $\mathcal{A}_N$ exhibits a similar rich structure as $\mathcal{A}$. Then we use the Homotopy invariance theorem of Brouwer degree (thanks to \cite{Carvalho2024, chepyzhov1992} that motivating us to apply such beautiful topology degree theory, in which they applied it to study the so-called unbounded attractors) to obtain $P_{\theta}\mathcal{A}_N=X_N$ for the approximate system \ref{touyingxitong}. In addition to the aforementioned interesting phenomena associated with $\mathcal{A}$, $\mathcal{A}_N$ exhibits also some interesting features, which are discussed Sections \ref{timu4} and \ref{timu5}.
	
The paper is organized as follows. In Section \ref{2}, we lay out notation and preliminary material. In Section \ref{timu4}, we adopt the Galerkin approximation to prove the existence and uniqueness of solutions for the projected system \ref{touyingxitong}. Especially, based on the proof process, we can identify the characteristics of this system as well as its connections with the original system. In Section \ref{timu5}, we successfully obtain $P_{\theta}\mathcal{A}_N=X_N$ by applying the Homotopy invariance theorem of Brouwer degree, and, we let $N$ tend to infinity to prove our main result, Theorem \ref{main-thm}. Finally, in Section \ref{s5}, we extend our analysis to discuss the dimension of $\mathcal{A}_r$ and $H_r$, and also present some discussions/observations to enrich the understanding of its architecture. 
	
\section{Preliminaries}\label{2}
	\subsection{Notation and function spaces}
	Let \( \mathcal{F} \) be the set of all vector-valued trigonometric polynomials with periodic domain \( \Omega=\mathbb{T}^2 \). We define the space of test functions which incorporates the divergence-free and zero-average condition to be
	\[
	\mathcal{V} := \big\{ \varphi \in \mathcal{F} : \nabla \cdot \varphi = 0 \text{ and } \int_{\mathbb{T}^2} \varphi(x)dx = 0 \big\}.
	\]
	For $ r \in [1, \infty] $ and positive integer $ m $, we use $ \mathbb{L}^r(\Omega)$  and $ \mathbb{H}^m(\Omega)$ to denote the usual vector-valued Lebesgue and Sobolev spaces over $ \Omega$, while $ \mathbb{H}^{-1}(\Omega)$ is the dual space of $ \mathbb{H}^1(\Omega)$. 
	We define \( \mathcal{H}\) and $V$ to be the closures of $ \mathcal{V} $ in $ \mathbb{L}^2(\Omega)$ and $\mathbb{H}^1(\Omega)$ respectively. 
	
	Observe from $\eqref{1.1}_2$ and $\eqref{1.1}_3$, if we assume that $\int_{\mathbb{T}^2} \theta_0(x) dx = 0$, then $\int_{\mathbb{T}^2} \theta(x,t) dx = 0$ for all $t \ge 0$, and also $\int_{\mathbb{T}^2} u(x,t) dx = 0$ for all $t \ge 0$ provided $\int_{\mathbb{T}^2} u_0(x) dx = 0$. Therefore, we can work in the spaces $\mathcal{H}\times \dot{L^2}$, where $\dot{L^2}$ is the usual scalar-valued Lebesgue space $L^2(\Omega)$  which incorporates the zero-average condition. And from here on, same as in the references \cite{adam2013}, to simplify notation, we write $\dot{L^2}$ as $L^2$, $\dot{H^1}(\Omega)$ as $H^1$ with periodic condition etc., and, use \( \|f\|_r \) to denote the \( L^r(\Omega) \) (or $ \mathbb{L}^r(\Omega)$) norm of \( f \). \( C \) denotes a generic constant which may change from line to line.
	
	We denote by $P_\sigma : \mathbb{L}^2(\Omega) \to \mathcal{H}$ the Leray-Helmholtz projection operator (i.e., the orthogonal projection onto divergence-free vector spaces), and the Stokes operator $A := -P_\sigma \Delta$ with domain $\mathcal{D}(A) := \mathbb{H}^2(\Omega) \cap V$. In our case of periodic boundary conditions, it is well known that $A = -\Delta$. We label the eigenvalues $\lambda_k$ of $A$ so that
	\begin{equation}
		0 < \lambda_1 \leq \lambda_2 \leq \cdots \leq \lambda_j \leq \cdots,\ \lambda_j \to +\infty,\ \text{as } j \to \infty,
	\end{equation}
	and let $\{\omega_j\}_{j=1}^\infty \subset \mathcal{D}(A)$ be the corresponding eigenfunctions.
	For \( 0 \leq r \leq s \leq 1 \), the Poincaré inequality asserts that:
	\[
	\lambda_1^{s - r} \|A^r u\|_2 \leq \|A^s u\|_2.
	\] 
	
	\begin{Definition}[\cite{Rapha2008,Biswas2017}]\label{ruojiedingyi}
		Suppose $(u_0, \theta_0) \in \mathcal{H} \times L^2$. A pair $(u, \theta)$ is said to solve (\ref{1.1}) if $(u(t), \theta(t)) \in \mathcal{H} \times L^2$ for each $t > 0$, $u$ and $\theta$ are both periodic and mean-free, and they satisfy

		\begin{equation}\label{jiedeshizi}
			\begin{aligned}
				&-\int_0^T \langle u(s), \Phi'(s) \rangle ds + \int_0^T \langle B(u(s), u(s)), \Phi(s) \rangle ds - \langle u_0, \Phi(\cdot, 0) \rangle \\
				&\qquad =-\nu\int_0^T \langle \nabla u(s) , \nabla \Phi(s) \rangle ds + \int_0^T \langle \theta(s) e_2, \Phi(s) \rangle ds,  \\
				&-\int_{0}^{T} \langle \theta(s), \phi'(s) \rangle ds + \int_{0}^{T} \langle u(s)\cdot\nabla\theta(s), \phi(s) \rangle ds = \langle \theta_0, \phi(\cdot, 0) \rangle 
			\end{aligned}
		\end{equation}
		for all mean-free, space-periodic scalar test functions $\phi(x, t) \in C^{\infty}(\Omega \times [0, T])$, such that $\phi(x, T) = 0$; and for all mean-free, space-periodic vector-valued test functions ${\Phi}(x, t) \in [C^{\infty}(\Omega \times [0, T])]^2$ such that $\nabla \cdot {\Phi}(\cdot, t) = 0$, ${\Phi}(\cdot, T) = 0$.
	\end{Definition}
	
	\begin{Theorem}[\cite{Rapha2008}]\label{thm2.4}
		Suppose \((u_0,\theta_0)\in \mathcal{H}\times L^{2}\). Then, for any \(T > 0\), the system (1.1) has a unique solution \((u,\theta)\) such that
		\begin{align*}
			{u}&\in C([0,T],\mathcal{H})\cap L^{2}([0,T],V), \qquad\theta\in C([0,T],L^{2})
		\end{align*}
		and
		\begin{align*}
			\frac{\partial{u}}{\partial t}&\in L^{2}([0,T],V'), \qquad \frac{\partial\theta}{\partial t}\in L^{4}([0,T],H^{-3/2}).
		\end{align*}
		If additionally, $u_0\in V$, then
		\begin{align*}
			{u}\in C([0,T],V)\cap L^{2}([0,T],\mathcal{D}(A)),~ \frac{\partial{u}}{\partial t}\in L^{2}([0,T],\mathcal{H}), \text{\rm and}
			~\frac{\partial\theta}{\partial t}\in L^{4}([0,T],H^{-1}).
		\end{align*}
		
	\end{Theorem}
		
	\subsection{Some a priori estimates and results in \cite{Biswas2017}}
	We recall in this subsection some results obtained in \cite{Biswas2017}.
	
	From Theorem \ref{thm2.4} we see that the solution operator of system \eqref{1.1} define a semigroup $\{S(t)\}_{t\geqslant 0}$ on $\mathcal{H}\times L^2$:
	\[
	S(t):~\mathcal{H}\times L^2\to \mathcal{H}\times L^2 ~\text{with}~(u_0,\theta_0)\to (u(t),\theta(t)),
	\]
	where $(u(t),\theta(t))$ is the solution, corresponding to the initial data $(u_0,\theta_0)$, given in Theorem \ref{thm2.4} at time $t$.
	
	Then, as shown in \cite{Biswas2017}, the corresponding solution $(u(t),\theta(t))$ satisfies
	\begin{equation}
		\|\theta(t)\|_2 = \|\theta_0\|_2 \quad \forall \ t \geq 0,
	\end{equation}
	and
	\begin{equation}
		\|u\|_2^2 \leq e^{-\nu\lambda_1 t}\|u_0\|_2^2 + \nu^2 G^2\big(1 - e^{-\nu\lambda_1 t}\big),
	\end{equation}
	where $G = \dfrac{\|\theta_0\|_2}{\nu^2\lambda_1}$ is the dimensionless, time-independent Grashof-type number of \eqref{1.1}.
	
	Thus, there exists a time $t_* = t_*(\|u_0\|_2)$ such that, for any $t > t_*$, $\|u(t)\|_2$ is in the ball of radius $2\nu G$ in $\mathcal{H}$. For example, $t_*$ can be taken as
	\begin{equation}
		t_*(\|u_0\|_2) = \frac{1}{\nu\lambda_1} \max\Big\{1,\log \frac{\|u_0\|_2^2}{3\nu^2 G^2}\Big\}.
	\end{equation}
	
	\begin{Lemma}[\cite{Biswas2017}]
		For each fixed $t \geq 0$ and any sequence $({u_{0,n}}, \theta_{0,n})$ weakly converging to $({u_0}, \theta_0)$ in $\mathcal{H}\times L^2(\Omega)$, the solution semigroup satisfies:
		\[
		S(t)({u_{0,n}}, \theta_{0,n}) \rightharpoonup S(t)({u_0}, \theta_0) \quad \text{weakly in } \mathcal{H} \times L^2(\Omega).
		\]
	\end{Lemma}
	
	\section{Boussinesq system with projection}\label{timu4}
	
	For each \( N \in \mathbb{N} \), let \( P_N: L^2 \to X_N \) denote the orthogonal projection onto the finite-dimensional subspace
	\[
	X_N := \operatorname{span}\{\varphi_1, \dots, \varphi_N\},
	\]
	where $\{\varphi_j\}_{j=1}^\infty \subset H^2$ are the eigenfunctions of $-\Delta$ in $H^2$, which is an orthogonal basis of $H^1$ and an orthonormal basis of $L^2$. We consider the Cauchy problem of 
	\begin{equation}\label{touyingxitong}
		\begin{cases}
			\partial_t u+(u\cdot \nabla)u  + \nabla p - \nu \Delta u  = \theta e_2, & (x,t)\in \Omega\times (0,\infty), \\
			\partial_t \theta+P_N(u\cdot \nabla\theta) = 0, & \\
			\text{div}\, u = 0, & \\
			u(x,0) = u_0(x),\quad \theta(x,0) = \theta_0(x). &
		\end{cases}
		\tag*{$(3.1)_N$}
	\end{equation}
	
	We will employ this projected system to investigate the original Boussinesq system.
		
	\begin{Definition}\label{solution}
		Suppose $(u_0, \theta_0) \in \mathcal{H} \times X_N$. A pair $(u, \theta)$ is said to solve \ref{touyingxitong} if $ (u(t), \theta(t)) \in \mathcal{H} \times X_N $ for each \( t > 0 \), $ u $ and $ \theta $ are both periodic and mean-free, and they satisfy, for any $T>0$,
		\begin{equation}\label{jie}
			\begin{aligned}
				&-\int_0^T \langle u(s), \Phi'(s) \rangle ds + \int_0^T \langle B(u(s), u(s)), \Phi(s) \rangle ds - \langle u_0, \Phi(\cdot, 0) \rangle \\
				&\qquad =-\nu\int_0^T \langle \nabla u(s) , \nabla \Phi(s) \rangle ds + \int_0^T \langle \theta(s) e_2, \Phi(s) \rangle ds,  \\
				&-\int_0^T \langle \theta(s),\phi'(s) \rangle ds + \int_0^T \langle u \cdot \nabla \theta,P_N \phi(s) \rangle ds = \langle \theta_0,\phi(\cdot, 0) \rangle 
			\end{aligned}
		\end{equation}
		for all mean-free, space-periodic scalar test functions \( \phi(x, y, t) \in C^\infty(\Omega \times [0, T]) \), such that \( \phi(x, y, T) = 0 \); and for all mean-free, space-periodic vector-valued test functions \( \Phi(x, y, t) \in [C^\infty(\Omega \times [0, T])]^2 \) such that \( \nabla \cdot \Phi(\cdot, t) = 0 \), \( \Phi(\cdot, T) = 0 \).
	\end{Definition}
	
	The existence of the solutions for the system \ref{touyingxitong} is derived with a standard Galerkin procedure from the energy inequality. Nevertheless, in order to more completely demonstrate that the dissipativity of the projected system \ref{touyingxitong} is the same as that of the original system \eqref{1.1}, and to ensure that certain boundedness properties can be more clearly presented in subsequent proofs, we would like to write it out in full here.
	
	\begin{Theorem} \label{touyingcunzaiweiyi}
		For each $N=1,2,\cdots$, let $(u_0,\theta_0)\in \mathcal{H}\times  X_N $, then the projection system \ref{touyingxitong} has a unique solution $(u, \theta)$ such that
		\[
		(u, \theta) \in L_{\text{loc}}^\infty(\mathbb{R}_+, \mathcal{H}) \cap L_{\text{loc}}^2(\mathbb{R}_+; V) \times L_{\text{loc}}^\infty(\mathbb{R}_+, X_N),
		\]
		\[
		\partial_t u \in L^2(0,T;V'),\quad \partial_t \theta \in L^2\bigl(0,T;H^{-2}\bigr),
		\]
		and
		\[
		u \in C\bigl([0,T];\mathcal{H}\bigr),\quad \theta \in C\bigl([0,T];L^2\bigr).
		\]
	\end{Theorem}
	\begin{proof} 
		For convenience, in the following proof, we set $\nu=1$.
		
		\textbf{Step 1}. Construct finite-dimensional approximations.
		
		As defined previously, let $\{\omega_j\}_{j=1}^\infty \subset \mathcal{D}(A)$ and $\{\varphi_j\}_{j=1}^\infty \subset H^2$ be the eigenfunctions of $A$ and $-\Delta$. Then $\{(\omega_j, \varphi_j)\}_{j=1}^\infty$ constitutes an orthogonal basis of $V \times H^1$ and an orthonormal basis of $\mathcal{H} \times L^2$.
		
		Fix now a positive integer $n\geq N$, we look for a pair $(u_n, \theta_n)$ of the form
		\begin{equation}\label{odedejie}
			u_n(t) := \sum_{j=1}^n d_n^j(t)\omega_j, \, \theta_n(t) := \sum_{j=1}^n \alpha_n^j(t)\varphi_j.
		\end{equation}
		We hope to select the coefficients $(d_n^j(t), \alpha_n^j(t))(0 \leq t \leq T, j = 1, \cdots, n)$ so that
		\begin{equation}\label{chuzhi}
			d_n^j(0) = \langle u_0, \omega_j \rangle, \, \alpha_n^j(0) = \langle \theta_0, \varphi_j \rangle \, (j = 1, \cdots, n)
		\end{equation}
		and
		\begin{equation}\label{ode}
			\begin{cases}
				\langle u_n', \omega_j \rangle + \langle  \nabla u_n(s) ,   \nabla\omega_j \rangle + \langle B(u_n, u_n), \omega_j \rangle = \langle \theta_n e_2, \omega_j \rangle, \\
				\langle\theta_n', \varphi_j \rangle + \langle u_n \cdot \nabla \theta_n, P_N\varphi_j \rangle = 0, \\
				(0 \leq t \leq T, j = 1, \cdots, n).
			\end{cases}
		\end{equation}

		According to standard existence theory for ordinary differential equations, there exists a unique local solution $(u_n, \theta_n)$ of the form (\ref{odedejie}) defined in an interval $[0, T_n]$, with $0 \leq T_n \leq T$. Furthermore, since $n \geq N$, the initial data satisfies $\alpha_n^j(0) = \langle \theta_0, \varphi_j \rangle = 0$ for all $j > N$. Consequently, from $\eqref{ode}_2$, $\theta_n$ reduces to a finite sum over the first $N$ basis functions:
		\begin{equation}\label{bijinjie}
			\theta_n(t) = \sum_{j=1}^N \alpha_n^j(t) \, \varphi_j.
		\end{equation}
		
		\textbf{Step 2}. A priori estimates about the functions $(u_n, \theta_n)$.
		
		Multiplying the equations (\ref{ode}) by $(d_n^j(t), \alpha_n^j(t))$, sum for $j = 1, \cdots, n$ and then we obtain the following estimates.
		
		About $\theta$: Thanks to \eqref{bijinjie}, we have
		\[
		\bigl\langle  \partial_t \theta_n, \theta_n \bigr\rangle + \bigl\langle u_n \cdot \nabla \theta_n, \theta_n \bigr\rangle = 0,
		\]
		then
		\[
		\frac{d}{dt} \|\theta_n\|_{L^2}^2  = 0.
		\]
		Consequently, we have \[
		\|\theta_n(t)\|_{L^2}^2 = \sum_{j=1}^N |\alpha_n^j(t)|^2 = \|\theta_0\|_{L^2}^2, \quad \forall t\in[0,T].
		\]
		Since $\theta_n(t)= \sum_{j=1}^N \alpha_n^j(t)\varphi_j \in X_N$, for $\theta$ we have the gradient estimate:
		\[
		\nabla \theta_n(t) = \sum_{j=1}^N \alpha_n^j(t)\nabla \varphi_j,
		\]
		and with the basis fixed, \(\sum_{j=1}^N \|\nabla \varphi_j\|_{L^2}^2\) is a finite constant depending on $N$. Then, by the triangle inequality and the Cauchy--Schwarz inequality, we obtain
		\begin{equation}\label{wendutiduyoujie}
			\|\nabla \theta_n(t)\|_{L^2} \leq \sum_{j=1}^N |\alpha_n^j(t)| \|\nabla \varphi_j\|_{L^2}
			\leq \|\theta_0\|_{L^2} \Big( \sum_{j=1}^N \|\nabla \varphi_j\|_{L^2}^2 \Big)^{1/2}
			\leq C(N)\|\theta_0\|_{L^2}.
		\end{equation}
		Hence $\{\theta_{n}\}$ is bounded in $L^\infty(0,T;H^1)$, and similarly, 
		\begin{equation}\label{3.14}
			\text{$\{\theta_{n}\}$ is bounded in $L^\infty(0,T;H^2)$}
		\end{equation}
		since $\{ \varphi_j\}_{j=1}^\infty \subset H^2$.
			
		About $u$: 
		The corresponding functional equations for $u_n$ is
		\begin{equation}\label{unde}
			\frac{du_n}{dt} +  A u_n + B(u_n, u_n) = \theta_n e_2 u_n,
		\end{equation}
		To bound $\|u_n\|_2$ we take the inner product of \eqref{unde} with $u_n$, obtaining
		\begin{equation}\label{0000}
			\frac{1}{2} \frac{d}{dt} \|u_n\|_2^2 + \|\nabla u_n(s)\|_2^2 = \int_\Omega \theta_n e_2 u_n dx.
		\end{equation}
		Using Young's inequality on the right-hand side and the Poincaré inequality, we obtain
		\begin{equation}\label{6}
			\begin{aligned}
				\frac{d}{dt}\|u_n\|_2^2 + \|\nabla u_n\|_2^2 \leq \dfrac{\|\theta_0\|_2^2}{\lambda_1}.
			\end{aligned}
		\end{equation}	
		Integrating both sides between $0$ and $t$ yields
		\begin{equation}\label{haosan}
			\|u_n(t)\|_2^2 +  \int_0^t \|\nabla u_n\|_2^2 ds \leq \|u_0\|_2^2 + \frac{\|\theta_0\|_2^2}{\lambda_1}  t.
		\end{equation}
		Since $\|u_n(0)\|_2 \le \|u_0\|_2$, we have the bounds
		\[
		\sup_{t\in[0,T]} \|u_n(t)\|_2^2 \le K = \|u_0\|_2^2 + \frac{\|\theta_0\|_2^2}{\lambda_1}  T
		\]
		and
		\[
		\int_0^T \|\nabla u_n\|_2^2 ds \le K
		\]
		uniformly in $n$. Thus 
		\begin{equation}\label{3.12}
			\text{$\{u_{n}\}$ is bounded uniformly (w.r.t. $n$) in $L^\infty(0,T;\mathcal{H})$ and $ L^2(0,T;V)$}.
		\end{equation}
		
	\textbf{Step 3}. Uniform bounds on the derivatives $\left\{ \left( \frac{du_n}{dt}, \frac{d\theta_n}{dt} \right) \right\}_{n \in \mathbb{N}}$.
		
		For any $\Phi \in V$ with $\|\Phi\|_{H^1} \leq 1$, we have
		\[
		\langle \partial_t u_n, \Phi \rangle = -\langle B(u_n, u_n), \Phi \rangle + \langle \Delta u_n,  \Phi \rangle + \langle \theta_n e_2, \Phi \rangle.
		\]
		Consequently
		\begin{equation}\label{gaoyijie}
			\begin{aligned}
				\langle \partial_t u_n, \Phi \rangle &\leq |\langle B(u_n, u_n), \Phi \rangle| + |\langle \nabla u_n, \nabla \Phi \rangle| + |\langle \theta_n, \Phi^2 \rangle| \\
				&\leq C \|u_n\|_2 \|\nabla u_n\|_2\|\Phi\|_{H^1}+ \| \nabla u_n\|_2 \| \nabla\Phi\|_2 + \|\theta_n\|_2 \|\Phi^2\|_2 \\
				&\leq C \|u_n\|_2 \|\nabla u_n\|_2 + \| \nabla u_n\|_2  + \|\theta_n\|_2,
			\end{aligned}
		\end{equation}
		since $\|\Phi\|_{H^1} \leq 1$. Thus
		\[
		\|\partial_t u_n\|_{H^{-1}} \leq C \|u_n\|_2 \|\nabla u_n\|_2 +  \| \nabla u_n\|_2 + \|\theta_n\|_2,
		\]
		and therefore
		\begin{equation}\label{uN dui t qiudao}
			\int_0^T \|\partial_t u_n\|_{H^{-1}}^2 ds \leq C \|u_n\|_{L^\infty(0,T;H)}^2 \int_0^T \|\nabla u_n\|_2^2 ds + \int_0^T \|\nabla u_n\|_2^2 ds + \|\theta_0\|_2^2 T.
		\end{equation}
		
		For any $\phi \in H^2(\Omega)$ with $\|\phi\|_{H^2} \leq 1$, we have
		\[
		\langle \partial_t \theta_n, \phi \rangle = -\langle u_n \cdot \nabla \theta_n,P_N \phi \rangle.
		\]
		Consequently
		\begin{align*}
			|\langle \partial_t \theta_n, \phi \rangle| &= |\langle u_n \cdot \nabla \theta_n, P_N\phi \rangle| = |\langle u_n \theta_n, \nabla (P_N\phi) \rangle| \\
			&\leq \|u_n\|_4 \|\theta_n\|_2 \|\nabla  (P_N\phi)\|_4 \\
			&\leq C \|\nabla u_n\|_2 \|\theta_n\|_2 \|\Delta  (P_N\phi)\|_2 \\
			&\leq C \|\nabla u_n\|_2 \|\theta_n\|_2\|\Delta  \phi\|_2 .
		\end{align*}
		Hence, for any test function $\phi$ with $\|\phi\|_{H^2} \le 1$,
		\[
		\|\partial_t \theta_n\|_{H^{-2}} \leq C \|\nabla u_n\|_2 \|\theta_n\|_2,
		\]
		and integrating in time yields
		\begin{equation}\label{wenduN dui t qiudao}
			\int_0^T \|\partial_t \theta_n\|_{H^{-2}}^2 ds \leq C \|\theta_0\|_2^2 \int_0^T \|\nabla u_n\|_2^2 ds.
		\end{equation}
		
		Consequently, the time derivatives satisfy \begin{equation}\label{3.20}
			(\partial_t u_n, \partial_t \theta_n) \text{ is bounded in } L^2_{\text{loc}}(\mathbb{R}_+; V^{'}) \times L^2_{\text{loc}}(\mathbb{R}_+; H^{-2}).
		\end{equation}
	
		\textbf{Step 4}. There exists $(u, \theta)$ which is a solution (in the sense of Definition \ref{jie}) to the system \ref{touyingxitong} with the initial data $(u_0, \theta_0)$.
		
		From \eqref{3.14}, \eqref{3.12} and \eqref{3.20}, we have
		\begin{equation}\label{jixianwendutiduyoujie}
			(u_n, \theta_n)  \text{ is bounded in } L_{\text{loc}}^\infty(\mathbb{R}_+, \mathcal{H}) \cap L_{\text{loc}}^2(\mathbb{R}_+; V) \times L_{\text{loc}}^\infty(\mathbb{R}_+, H^2)
		\end{equation}
		$$	(\partial_t u_n, \partial_t \theta_n) \text{ is bounded in } L^2_{\text{loc}}(\mathbb{R}_+; V^{'}) \times L^2_{\text{loc}}(\mathbb{R}_+; H^{-2}).$$
		Therefore, there is a subsequence, still denoted by \( \{(u_n, \theta_n)\} \), such that, or any (fixed) $ T > 0$,
		\begin{equation}\label{week1   star}
			(u_n, \theta_n) \stackrel{*}{\rightharpoonup} (u, \theta) \text{ in } L^\infty(0, T; \mathcal{H}) \times L^\infty(0, T; H^2),
		\end{equation}
		\begin{equation}\label{u zai}
			u_n \rightharpoonup u \text{ in } L^2(0, T; V), 
		\end{equation}
		\begin{equation}\label{wendu dui t qiudao suoshu d kongjian}
			(\partial_t u_n, \partial_t \theta_n) \rightharpoonup (\partial_t u, \partial_t \theta) \text{ in } L^2(0, T; V^{'}) \times L^2(0, T; H^{-2}).
		\end{equation}
		So we have \( u \in C([0, T]; \mathcal{H}) \) by Lions-Magenes Lemma.
		Moreover, by the Aubin-Lions lemma, there is a subsequence, still denoted by \( \{(u_n, \theta_n)\} \), such that
		\begin{equation}
			(u_n, \theta_n) \to (u, \theta) \text{ in } L^2(0, T; \mathcal{H}) \times L^2(0, T; H^{-1}),
		\end{equation}
		\begin{equation}\label{week}
			\theta_n \to \theta \text{ in }  C(0, T; L^2).
		\end{equation}
		Furthermore, convergence (\ref{week1   star})-(\ref{week}) allow us to pass to the limit in the equations for \( (u_n, \theta_n) \) to find that \( (u, \theta) \) is a solution of \ref{touyingxitong}.
		
		Note that \( (u_n, \theta_n) \) satisfy \eqref{jiedeshizi} for appropriate test functions \( (\Phi, \phi) \) chosen as in Definition \ref{solution}. The convergence of the linear terms is straightforward. It remains to show the convergence of the remaining non-linear terms. We have
		\begin{align*}
			&\Big| \int_0^T \langle B(u_n, u_n), \Phi \rangle - \langle B(u, u), \Phi \rangle ds \Big| \\
			&= \Big| \int_0^T \langle B(u_n - u, u_n), \Phi \rangle + \langle B(u, u_n - u), \Phi \rangle ds \Big| \\
			&\leq \Big| \int_0^T \langle B(u_n - u, \Phi), u_n \rangle ds \Big| + \Big| \int_0^T \langle B(u, \Phi), u_n - u \rangle ds \Big| \\
			&\leq \|u_n - u\|_{L^2(0,T;H)} \|\nabla \Phi\|_{L^2(0,T;L^\infty(\Omega))} \|u_n\|_{L^\infty(0,T;H)} \\
			&\qquad + \|u\|_{L^\infty(0,T;H)} \|\nabla \Phi\|_{L^2(0,T;L^\infty(\Omega))} \|u_n - u\|_{L^2(0,T;H)} \to 0,
		\end{align*}
		as $n \to \infty$, since $u_n \to u$ in $L^2(0,T;H)$ and $u_n$ is uniformly (w.r.t. $n$) bounded in $L^\infty(0,T;H)$.
		
		And 
		\begin{align*}
			&\Big| \int_0^T \langle u_n \cdot \nabla \theta_n, P_N\phi \rangle - \langle u \cdot \nabla \theta,P_N \phi \rangle ds \Big| \\
			&= \Big| \int_0^T \langle (u_n - u) \cdot \nabla \theta_n,P_N \phi \rangle + \langle u \cdot \nabla (\theta_n - \theta),P_N \phi \rangle ds \Big| \\
			&\leq \Big| \int_0^T \langle (u_n - u) \cdot \nabla (P_N\phi), \theta_n \rangle ds \Big| + \Big| \int_0^T \langle u \cdot \nabla(P_N\phi), \theta_n - \theta \rangle ds \Big| \\
			&\leq \|u_n - u\|_{L^2(0,T;H)} \|\nabla P_N \phi\|_{L^2(0,T;L^\infty(\Omega))} \|\theta_n\|_{L^\infty(0,T;L^2)} \\
			&\quad + \|u\|_{L^2(0,T;\mathcal{H})} \|\Delta P_N\phi\|_{L^\infty(0,T;L^\infty(\Omega))} \|\theta_n - \theta\|_{L^2(0,T;H^{-1})}\\
			&\quad + \|u\|_{L^2(0,T;V)} \|\nabla P_N\phi\|_{L^\infty(0,T;L^\infty(\Omega))} \|\theta_n - \theta\|_{L^2(0,T;H^{-1})} \to 0,
		\end{align*}
		as $n \to \infty$, since $(u_n, \theta_n) \to (u, \theta)$ in $L^2(0,T;H) \times L^2(0,T;H^{-1})$ and $\theta_n$ is uniformly bounded in $L^\infty(0,T;L^2)$.
	
		Then we have
		\begin{equation}
			\left\{
			\begin{aligned}
				&-\int_0^T \langle u(s), \Phi'(s) \rangle ds + \int_0^T \langle B(u(s), u(s)), \Phi \rangle ds - \langle u_0, \Phi(\cdot, 0) \rangle \\
				&= -\int_0^T \langle \nabla u(s) , \nabla \Phi(s) \rangle ds + \int_0^T \langle \theta(s) e_2, \Phi(s) \rangle ds, \\
				&-\int_0^T \langle \theta(s), \phi'(s) \rangle ds + \int_0^T \langle u(s) \cdot \nabla \theta(s), P_N\phi(s) \rangle ds = \langle \theta_0,  \phi(\cdot, 0) \rangle,
			\end{aligned}
			\right.
		\end{equation}
		since \( (u_{0n}, \theta_{0n}) \to (u_0, \theta_0) \) in \( \mathcal{H} \times L^2 \).
		
		Due $\theta_n(t) = \sum_{j=1}^N \alpha_n^j(t)\varphi_j \in X_N$ is finite dimensional, and\[
		\sup_{t \in [0,T]} \|\theta_n(t) - \theta(t)\|_{2} \to 0 \quad (n \to \infty),
		\] then $\theta(t)$ is also in $X_N$. Moreover, for all $t\geq0$, the $L^2$-norm is preserved along the sequence:
		\begin{equation}\label{wendububian}
			\| \theta(t)\|_{2}=\|\theta_n(t)\|_{2} =\| \theta_0\|_{2}.
		\end{equation}
		Indeed, $\theta(t)$ admits the basis expansion
		\begin{equation}\label{bijinxulie}
			\theta(t)=\sum_{j=1}^N \alpha^j(t)\varphi_j, \quad \alpha_n^j(t)\to \alpha^j(t), n\to \infty.
		\end{equation}
		A gradient estimate then follows, as that in \eqref{wendutiduyoujie}, 
		\begin{equation}\label{dituyoujie}
			\|\nabla \theta(t)\|_{2} \leq \sum_{j=1}^N |\alpha^j(t)| \|\nabla \varphi_j\|_{2}
			\leq C(N)\|\theta_0\|_{2}.
		\end{equation}
		
		Therefore, $ (u, \theta) \in \mathcal{H}\times X_N$ is a solution to the system \ref{touyingxitong} with the initial data $ (u_0, \theta_0) \in \mathcal{H}\times X_N $.
		
		\textbf{Step 5.} Uniqueness.
		
		The uniqueness of \ref{touyingxitong} is simpler than that for \eqref{1.1} since $X_N$ is finite-dimension.
		
		Suppose $( {u}_1,\theta_1, p_1)$, $( {u}_2,\theta_2, p_2)$ are two solutions of the 2D Boussinesq projection system with the same initial data $( {u}_0,\theta_0)$. Let $\omega = {u}_1 - {u}_2$ and $\xi = \theta_1 - \theta_2$. Then $(\omega, \xi)$ satisfies the following equations:
		\begin{align}
			\partial_t \omega + {u}_1 \cdot \nabla \omega + \omega \cdot \nabla {u}_2 + (-\Delta) \omega &= -\nabla(p_1 - p_2) + \xi e_2, \label{3.32}\\
			\partial_t \xi+ P_N({u}_1 \cdot \nabla \xi+\omega \cdot \nabla \theta_2 ) &= 0.\label{3.33}
		\end{align}
		Taking the $L^2$-inner product of \eqref{3.32} with \( w \) and noting that \( \nabla \cdot u_1 = 0 \), we obtain that
		\[
		\frac{1}{2} \frac{d}{dt} \|w\|_{2}^2 + \left\|\nabla w \right\|_{2}^2 = -\langle (w \cdot \nabla) u_2, w \rangle + \langle \xi e_2, w \rangle.
		\]
		The right-hand side terms can be estimated as follows:
		$$
		\left| \langle (w \cdot \nabla) u_{2}, w \rangle \right| \leq \| \nabla u_2 \|_{2} \| w \|_{4}^2
		\leq C \| \nabla u_2 \|_{2}^2 \| w \|_{2}^2 + \frac{1}{4} \| \nabla w \|_{2}^2,
		$$ 
		$$ \left| \langle \xi e_2, w \rangle \right| \leq \frac{1}{2} \|\xi\|_{2}^2 + \frac{1}{2} \|w\|_{2}^2.
		$$
		
		Taking the inner product of \eqref{3.33} with \( \xi \), yields
		\[
		\frac{1}{2}\frac{d}{dt}\|\xi\|_{2}^2 + \langle P_N({u}_1 \cdot \nabla \xi), \xi\bigr\rangle + \langle P_N(w \cdot \nabla \theta_2), \xi\rangle = 0.
		\]
		Since $P_N$ is self-adjoint and $\xi \in X_N$, we have $\langle P_N({u}_1 \cdot \nabla \xi), \xi\rangle = \langle {u}_1 \cdot \nabla \xi, \xi\rangle$. The incompressibility condition $\nabla \cdot {u}_1 = 0$ allows us to integrate by parts, giving $\langle {u}_1 \cdot \nabla \xi, \xi\rangle = 0$. Furthermore, as $P_N \xi = \xi$, we know that $\|\xi\|_{H^1}\leqslant C(N)\|\xi\|_{2}$, and the coupling term simplifies to $\langle P_N(w\cdot \nabla \theta_2), \xi\rangle = \langle w \cdot \nabla \theta_2, \xi\rangle$, and from \eqref{dituyoujie}, we finally have that
		\[
		|\langle w \cdot \nabla \theta_2, \xi\rangle|
		\leq \|w\|_{2}\|\nabla \theta_2\|_{4}\|\xi\|_{4}
		\leq C\|w\|_{2}\|\nabla \theta_2\|_{H^1}\|\xi\|_{H^1}\\
		\leq C(N,\|\theta_{0}\|_{2})\|w\|^2_{2}.
		\]
		
		Consequently,
		\begin{equation}\label{wenduguji}
			\begin{aligned}
				\frac{1}{2}\frac{d}{dt}\|\xi\|_{2}^2 = -\langle w \cdot \nabla \theta_2, \xi\rangle
				\leq C(N,\|\theta_{0}\|_{2})\|w\|^2_{2}.
			\end{aligned}
		\end{equation}
		Therefore, we have that
		\[
		\frac{d}{dt} (\|w\|_{2}^2 +\|\xi\|_{2}^2)\leq  C(N,\|\theta_{0}\|_{2})(1+ \| \nabla u_2 \|_{2}^2)(\|w\|_{2}^2 +\|\xi\|_{2}^2),
		\]
		we can rewrite this as 
		\begin{equation}\label{weiyixing}
			\|w(t)\|_2^2 + \|\xi(t)\|_2^2 \leq \left(\|w(0)\|_2^2 + \|\xi(0)\|_2^2\right) \exp\left(  C(N,\|\theta_{0}\|_{2}) \int_{0}^{t} \left(1 + \|\nabla u_2(s)\|_2^2\right) ds \right)
		\end{equation}
		which, combining with a similar estimate for $u_2$ as \eqref{haosan}, allows us to obtain the uniqueness immediately.
	\end{proof}
	
	From the proof above, we have the following estimates about the solution of \ref{touyingxitong}.
	
	\begin{Corollary}\label{yyyyyyyoujie}
		The solutions $(u(t),\theta(t))$ of system \ref{touyingxitong} satisfy
		$$\|u\|_{L^\infty(0,T;\mathcal{H})}\leq  C(\|\theta_0\|_2,\|u_0\|_2),$$
		$$\|u\|_{L^2(0,T;V)}\leq  C(T,\|\theta_0\|_2,\|u_0\|_2),$$
		$$\|\theta\|_{L^\infty(0,T;L^2)}=\|\theta_0\|_2,$$
		\begin{equation}\label{H2youjie}
			\|\theta\|_{L^\infty(0,T;H^2)}\leq C(N)\|\theta_0\|_{2},
		\end{equation}
		$$\|\partial_t u\|_{L^2(0,T;V^{'})}\leq  C(T,\|\theta_0\|_2,\|u_0\|_2),$$
		\begin{equation}\label{Hfu2}
			\|\partial_t \theta\|_{L^2(0,T;H^{-2})}\leq  C(T,\|\theta_0\|_2,\|u_0\|_2).
		\end{equation}
	\end{Corollary}

	From Theorem \ref{touyingcunzaiweiyi}, for each $N\in \mathbb{N}$, we can define an operators semigroup $\{S_N(t)\}_{t\geqslant 0}$ as follows:
	\[
	S_N(t): \mathcal{H} \times X_N \to \mathcal{H} \times X_N, \quad 0 \leq t < \infty,
	\]
	\[
	S_N(t)(u_0, \theta_0) := (u(t), \theta(t)) \quad \text{for any } (u_0, \theta_0) \in \mathcal{H} \times X_N,
	\]
	where $(u(t), \theta(t))$ is the solution of \ref{touyingxitong} for any $t \geq 0$. 
	
	This semigroup is weakly continuous, as can be shown by adapting the proof of Proposition 5.3 in \cite{Biswas2017}, so we do not repeat the proof here.
	\begin{Lemma}
		For each (fixed) $N\in \mathbb{N}$, fixed $t\geq0$ and a sequence $({u}_{0,n}, \theta_{0,n}) \stackrel{wk}{\rightharpoonup} ({u}_0, \theta_0)$, we have
		\[
		S_N(t)({u}_{0,n}, \theta_{0,n}) \stackrel{wk}{\rightharpoonup} S_N(t)({u}_0, \theta_0). 
		\]
		In particular, for each fixed $t$, the map $S_N(t): B \to \mathcal{H} \times X_N$ is weakly continuous, where $B$ is a bounded subset of $\mathcal{H} \times X_N$.
	\end{Lemma}
	
	Therefore, following the approach of \cite{Biswas2017}, one can establish the existence of a weak $\sigma$-attractor for system \ref{touyingxitong}. The sole distinction from the cited work lies in the function space chosen for the temperature component; however, the dissipation estimates remain identical to those of the original system. Consequently, we can obtain the similar results about the weak $\sigma$-attractor:
	\begin{Definition}\label{touyingxitongxiyinzi}
		For each (fixed) $N\in \mathbb{N}$,  defining the \textit{weak $\sigma$-attractor} \( \mathcal{A}_N\) of the semi-dissipative system \ref{touyingxitong} to be the set of all \( ({u}_0, \theta_0) \in \mathcal{H} \times X_N \) with the property that:
		\begin{enumerate}
			\item[(i)] There exists a global trajectory \( ({u}(t), \theta(t)) \) defined for all \( t \in \mathbb{R} \) such that \( ({u}(t), \theta(t)) \) belongs to \(  \mathcal{H} \times X_N \) and solves \ref{touyingxitong} for all \( t \in \mathbb{R} \), and moreover \( ({u}(0), \theta(0)) = ({u}_0, \theta_0) \);
			\item[(ii)] The trajectory \( ({u}(t), \theta(t)) \) is globally bounded, i.e., the set \( \{ ({u}(t), \theta(t)) : t \in \mathbb{R} \} \) is bounded in \(  \mathcal{H} \times X_N \).
		\end{enumerate}
	\end{Definition}
	And, similarly, we have
	\begin{Definition}\label{touyingxitongxiyinzir}
		The local attractor of \ref{touyingxitong} at level $r$, denoted by $\mathcal{A}_{Nr}$, is defined as the $\omega$-limit set of the positively invariant set $$\mathcal{B}_{Nr} := \bigl\{ ({u}_0,\theta_0) \in \mathcal{H} \times X_N : \|\theta_0\|_{L^2} \leq r,\ \|{u}_0\|_{L^2} \leq R(r) \bigr\},\quad R(r) := \frac{2r}{\nu \lambda_1}.$$
		That is,
		\begin{equation}\label{anr}
			\mathcal{A}_{Nr} := \omega^{wk}(\mathcal{B}_{Nr}) := \bigcap_{\tau\ge 0} \overline{\vphantom{\bigcap} \bigl\{ S_N(t)({u}_0,\theta_0) : t > \tau,\ ({u}_0,\theta_0)\in \mathcal{B}_{Nr} \bigr\}}^{wk},
		\end{equation}
		where the closure is taken in the weak topology of $\mathcal{H} \times X_N$.
	\end{Definition}
	Obviously, we have
	\begin{equation}
		\begin{aligned}
			\mathcal{A}_{Nr} = \bigl\{ (u_a,\theta_a) :\ &\exists\ t_n \to \infty \text{ and } (u_{0n},\theta_{0n}) \in \mathcal{B}_{Nr}, \\
			&\text{such that } S_N(t_n)(u_{0n},\theta_{0n}) \rightharpoonup (u_a,\theta_a) \text{ in } \mathcal{H} \times X_N\bigr\}.
		\end{aligned}
	\end{equation}

	Then, note that $X_N$ is finite-dimension,  as that in \cite{Biswas2017, sun2019}, we can obtain the following results immediately.
	\begin{Theorem}
		For each (fixed) $N\in \mathbb{N}$, the semigroup $\{S_N(t)\}_{t\geqslant 0}$ has a weak $\sigma$-attractor $\mathcal{A}_N$ in the sense of the Definition \ref{touyingxitongxiyinzi}, and has the following properties:
		\begin{enumerate}
			\item[(i)] The relation $\displaystyle \mathcal{A}_N = \bigcup_{r\ge 0} \mathcal{A}_{Nr} = \bigcup_{\substack{r\ge 0,\,r\in\mathbb{Q}}} \mathcal{A}_{Nr}$ holds, where $\mathcal{A}_{Nr}$ is defined by (\ref{anr}), and $\mathcal{A}_{Nr}$ is compact in $\mathcal{H}\times L^2$.
			
			\item [(ii)] The set $\mathcal{A}_N$ attracts all $\mathcal{H}\times X_N$-bounded sets in the strong topology of $V \times H^m(\Omega), m\in\mathbb{N}$.
		\end{enumerate}
	\end{Theorem}
	
	\begin{Remark}
		Obviously, $(0,0)\in \mathcal{A}_{Nr}\cap\mathcal{A}_{N'r}$ whence $N\leq N'$. However, there may be some other fixed points in the intersection, for example, when the basis of $X_N$ contains some fixed points of the $\theta$-equation, then some steady state solution $\{(0,\theta^{v})\}$ will also in $\mathcal{A}_{Nr}\cap\mathcal{A}_{N'r}$. However, whether there exists a common continuous evolutionary trajectory connecting $\mathcal{A}_{Nr}$ and $\mathcal{A}_{N'r}$ is unknown.
	\end{Remark}
	%
			\begin{Remark}
				From [Lemma 5.2, \cite{Biswas2017}], we know that if $(u,\theta)\in\mathcal{A}_{Nr}$, then there exists a dimensionless, absolute constant $C>0$ such that
				\[
				\|u\|_{\mathbb{H}^1} \leq C \lambda_1^{1/2}R(r) \quad \forall\, t \in \mathbb{R}.
				\] 
				Since $X_N$ is finite-dimension, combining with the result [Theorem 3.3, \cite{sun2019}] that the attraction properties about $u$ is in the sense of the strong topology of $V$, we know that the attraction in system \ref{touyingxitong} can be the strong topology of $\mathcal{H}\times L^2$. Actually, $\mathcal{A}_{Nr}$ is a compact subset of $V \times H^m(\Omega)$ ($m\in\mathbb{N}$) and the closure of $\mathcal{A}_{Nr}$ in \eqref{anr} can be taken in the strong topology of $V \times H^m(\Omega)$. Furthermore, by the parabolic regularity of the 2D Navier–Stokes equations and the smoothness of the eigenfunctions of the Laplacian on the periodic domain, we can have  $\mathcal{A}_N$ is infinitely smooth.
			\end{Remark}

			
\section{Structures of $\mathcal{A}_N$ and $\mathcal{A}$}\label{timu5}
			
\subsection{Project $\mathcal{A}_N$ onto $X_N$}
			
			We next perform the projection $P_{\theta}$ on the attractor $\mathcal{A}_N$, recall that $P_{\theta}({u},\theta) := \theta$. Owing to the invariant property of the temperature $L^2$-norm for this system, each component $\mathcal{A}_{Nr}$ in the internal decomposition of the weak $\sigma$-attractor $\mathcal{A}_N$ can be regarded as an independent local attractor, thus we project these small attractors $\mathcal{A}_{Nr}$ directly.
			
			We first deduce the following continuity as the preliminary. 
\begin{Lemma}\label{lianxu}
Define $B = \{ \theta \in  X_N : \|\theta\|_{2} < r + 1 \} $ is an open set of $X_N$. For each fixed \( t > 0 \) and define the mapping $[0, 1] \times \overline{B} \ni (\eta, x) \mapsto P_\theta S_N(\eta t)(0,x) \in X_N$, then the mapping is continuous from $[0,1] \times \overline{B}$ into $X_N$. 
\end{Lemma}
				\begin{proof}
					
					Let $ R^{+}\times \overline{B}\ni (t_n,\theta_{0n})\to (t,\theta_{0})$ in $R^{+}\times \overline{B}$ as $n\to\infty$, and denote $$\theta_{n}(t_n)=P_{\theta}(u_n(t_n),\theta_{n}(t_n)):=P_{\theta}S_N(t_n)(0,\theta_{0n}),$$
					$$\theta(t)=P_{\theta}(u(t),\theta(t)):=P_{\theta}S_N(t)(0,\theta_{0}).$$ We aim to show $\theta_{n}(t_n)\to \theta(t)$ as $n\to \infty$. By the triangle inequality
\[
	\|\theta_n(t_n) - \theta(t)\|_{2} \leq \|\theta_n(t_n) - \theta_n(t)\|_{2} + \|\theta_n(t) - \theta(t)\|_{2}:=I_1+I_2.
\]
The convergence $I_2\to 0$ as $n\to \infty$ analogous to the proof of uniqueness in Theorem \ref{touyingcunzaiweiyi}, and by \eqref{weiyixing}, which immediately implies that if $\theta_n(0)\to \theta(0)$, then $\theta_n(t)\to\theta(t)$ for each $t\geq 0$.
				
					It remains to prove that  $I_1\to 0$ as $n\to \infty$.
					Without loss of generality, we assume that $t_n \in [0,T]$.	Then, we have \begin{equation}
						\begin{aligned}
							\|\theta_{n}(t_n)-\theta_{n}(t)\|_{H^{-2}}&\leq \Big|\int_{t}^{t_n} \|\frac{d}{d\tau}\theta_n(\tau) \|_{H^{-2}}  d\tau\Big|\\
							&\leq \Big|\int_{t}^{t_n} \|\frac{d}{d\tau}\theta_n(\tau) \|_{H^{-2}} ^2 d\tau\Big|^{\frac{1}{2}}|t_n-t|^{\frac{1}{2}}\\
							&\leq \|\partial_t \theta_n\|_{L^2(0,T;H^{-2})}|t_n-t|^{\frac{1}{2}}.
						\end{aligned}
					\end{equation}
					Using \eqref{H2youjie}, \eqref{Hfu2} and the interpolation inequality, we have
					\begin{equation}
						\|\theta_{n}(t_n)-\theta_{n}(t)\|_{2}\leq 	\|\theta_{n}(t_n)-\theta_{n}(t)\|_{H^{-2}}^{\frac{1}{2}}\|\theta_{n}(t_n)-\theta_{n}(t)\|_{H^{2}}^{\frac{1}{2}}\leq 
						C(N,\|\theta_0\|_2,T)|t_n-t|^{\frac{1}{4}}
					\end{equation}
					
					Therefore,
					\[
					\|\theta_n(t_n) - \theta_n(t)\|_{2} \leq C(N,\|\theta_0\|_{2},T) |t_n-t|^{\frac{1}{4}}\to 0 ,\ \text{as}\ n\to \infty.
					\]
				\end{proof}

			In the following, we will use the Homotopy invariance of Brouwer degree to obtain the expected existence.
			\begin{Theorem}\label{touyingdingli}
				For each fixed $N\in\mathbb{N}$, about system \ref{touyingxitong}, for every $r\geq0$, we have $P_\theta\mathcal{A}_{Nr}=\{\theta\in X_N: \|\theta\|_{2}\leq r\}$, which implies $P_\theta\mathcal{A}_N=X_N$.
			\end{Theorem}
			\begin{proof}
				It is obviously that $P_\theta\mathcal{A}_{Nr}\subseteq P_\theta\mathcal{B}_{Nr}=\{\theta\in X_N: \|\theta\|_{2}\leq r\}$.
				
				In the following we will establish the converse inclusion.	Define \( B = \{ \theta \in X_N : \|\theta\|_{2} < r + 1 \}, \) an open subset in $X_N$.
				For every $t\geq0$, $P_\theta S_N( t) (0, \cdot): B\to X_N$ is continuous by Lemma \ref{lianxu}. Then for any $ p\in \{\theta\in X_N: \|\theta\|_{2}\leq r\}$, we can define the Brouwer degree for the mapping $P_{\theta}S_N(t)(0,\cdot)$.
				
				For every $t\geq0$, we consider the mapping $	H_t(\eta,x)=P_\theta S_N(\eta t) (0,x);
				[0, 1] \times \overline{B}  \mapsto X_N$, then from Lemma \ref{lianxu} we have $H_t(\eta,x)$ is continuous with respect to $(\eta,x)$, and for all   $p\in\{\theta\in X_N: \|\theta\|_{2}\leq r\}$, \( p \notin P_{\theta}S_N(\eta t)(0,\partial B) \) for every \( \eta \in [0, 1] \).
				
				When $\eta=0$, we have $P_{\theta}S_N(\eta t)(0,x)=Id(x)=x$, and $\deg(Id, B, p) = 1$. Hence, by the homotopy invariance of the Brouwer degree, cf. [\cite{oregan2006}, Theorem 1.2.6 (3)] we deduce that $\deg(P_{\theta}S_N(t)(0,\cdot), B, p) = 1$, consequently, cf. [\cite{oregan2006}, Theorem 1.2.6 (2)], there exists $x_t \in B \) such that \( P_{\theta}S_N(t)(0,x_t) = p$.
				
				Let \( t_n \to \infty \), there exists a sequence $x_n \in B$ such that $P_{\theta}S_N(t_n)(0,x_n) \equiv p$. Hence we can find $ q_n \in \mathcal{H} $ such that $ y_n = (0,p) + (q_n,0) = S_N(t_n)(0,x_n) \in S_N(t_n)\mathcal{B}_{N_{r+1}}$. Actually $(0,x_n)\in \mathcal{B}_{Nr}$ and $y_n\in S_N(t_n){\mathcal{B}_{Nr}}$, since $(0,p)\in \mathcal{B}_{Nr}$. 
				
				Because $\{y_n\}$ is bounded, we can find a subsequence $y_n\rightharpoonup y$, by the equivalent definition of $\omega^{wk}(\mathcal{B}_{N{r}})$, we can know $y\in \omega^{wk}(\mathcal{B}_{N{r}})$ and $P_{\theta}(y)=P_{\theta}(y_n)=p$.
			\end{proof}
			
			\begin{Remark}
				In fact, if we define $\widetilde{\mathcal{B}}_{Nr} := \{(u_0,\theta_0) \in \mathcal{H} \times X_N: \|\theta_0\|_2 =r,\, \|u_0\|_2 \leq R(r)= \frac{2r}{\nu\lambda_1}\}$, then the proof of Theorem \ref{touyingdingli} can directly yields $P_{\theta}\omega^{wk}(\widetilde{\mathcal{B}}_{Nr})=\{\theta\in X_N: \|\theta\|_{2}= r\}$.
			\end{Remark}
			\begin{Remark}\label{more interesting}
				Inspection of the proof of Lemma \ref{lianxu} and Theorem \ref{touyingdingli} reveals that for every $N\in \mathbb{N}, r\geq0$, if we fix $u_0 \in P_u\mathcal{B}_{Nr}$ (where $P_u({u},\theta) :=u$) and set the mapping $	H_t(\eta,x)=P_\theta S_N(\eta t) (u_0,x);
				[0, 1] \times \overline{B}  \mapsto X_N$, we still can obtain the result of Theorem \ref{touyingdingli}.
				That is, for each fixed $u_0 \in \mathcal{H}$ with $\|u_0\|_2 \le R(r)$, set $\widetilde{\mathcal{B}}_{Nr}^{u_0}:=\{(u_0,\theta):\theta \in  X_N, \|\theta\|_{2}= r\} $, then we still have  $P_\theta\omega^{wk}(\widetilde{\mathcal{B}}_{Nr}^{u_0})=\{\theta\in X_N: \|\theta\|_{2}= r\}$.
				
				This suggests that: if initial temperatures $\theta_{0}$ taken from the sphere $\mathbb{S}_{Nr} := \{\theta \in X_N : \|\theta\|_2 = r\}$ and fix $u_0$ satisfying $\|u_0\|_2 \leq R(r)$, the temperature component fills the entire sphere $\mathbb{S}_{Nr}$ again as $t \to \infty$. 
			\end{Remark}

			\subsection{$\mathcal{A}_{Nr}$ approximates $\mathcal{A}_{r}$ as $N\to \infty$}
			For system \eqref{1.1}, fix $r \geq 0$, we can get local attractor $\mathcal{A}_{r}$ at level $r$ as defined in Definition $\ref{yuanxitongxiyinzir}$;
			For system \ref{touyingxitong}, fix $N \in \mathbb{N}$, $r \geq 0$, we can get local attractor $\mathcal{A}_{Nr}$ at level $r$ as defined in Definition \ref{touyingxitongxiyinzir}. We now state a key theorem establishing the connection between $\mathcal{A}_{Nr}$ and the global attractor $\mathcal{A}_{r}$.
			
\begin{Theorem}\label{bijinyinli}
For every fixed $r>0$, let $(u_{N0},\theta_{N0})\in\mathcal{A}_{Nr}$ ($N=1,2,\cdots$). If $(u_{N0},\theta_{N0})\rightharpoonup(u_0,\theta_{0})$ in $\mathcal{H}\times L^2$ as $N\to \infty$, then $(u_0,\theta_{0})\in \mathcal{A}_r$.
\end{Theorem}
\begin{proof}	
			
				\textbf{Step 1: Construct the trajectory segment.}
				
				For each $N\in\mathbb{N} $, there is a bounded complete trajectory
				\[
				\{(u_N(t),\theta_{N}(t)):t\in\mathbb{R}\}\subseteq\mathcal{A}_{Nr}\subseteq\mathcal{B}_{Nr}\subseteq\mathcal{B}_r\tag{$*$}\label{xing}
				\]
				such that $(u_{N}(0),\theta_{N}(0))=(u_{N0},\theta_{N0})$.

				Fix $T\geq2$. We take a segment of the trajectory:
				\[
				\{(u_{N}(t-T),\theta_{N}(t-T)): t\in [0,T+1]\}.
				\] 
				
				From Corollary \ref{yyyyyyyoujie}, we have
				\begin{equation}\label{Nyoujie}
					\{(u_{N}(\cdot-T),\theta_{N}(\cdot-T))\} \text{ is bounded in } L^\infty(0,T+1;\mathcal{H}) \cap L^2(0,T+1;V)\times L^\infty(0,T+1;L^2),
				\end{equation}
				\begin{equation}\label{partial un}
					(\partial_t u_N(\cdot-T), \partial_t \theta_N(\cdot-T)) \text{ is bounded in } L^2(0,T+1; V^{'}) \times L^2(0,T+1; H^{-2}).
				\end{equation}
				
				From the Aubin-Lions lemma, there is a subsequence, to simplify the presentation, we still denoted by \( \{(u_N(\cdot-T), \theta_N(\cdot-T))\} \), such that
				\begin{equation}\label{qiang}
					(u_N(\cdot-T), \theta_N(\cdot-T)) \to (u_{\infty}(\cdot-T), \theta_{\infty}(\cdot-T)) \text{ in } C(0, T+1; V^{'})\times C(0, T+1; H^{-1})
				\end{equation}
				and by Banach–Alaoglu theorem, without loss of generality, we assume that
				\begin{equation}\label{bijinguanjian}
					(u_N(-T), \theta_N(-T)) {\rightharpoonup}(u^{(1)}_0,\theta^{(1)}_{0})\in 
					\mathcal{B}_r\quad \text{ in $\mathcal{H}\times L^2$}.
				\end{equation}
				
				Next, we will show that $  (u_{\infty}(\cdot-T), \theta_{\infty}(\cdot-T))$ satisfies \eqref{jiedeshizi} with $t\in[0,T+1]$ for all mean-free, space-periodic scalar test functions $\phi(x, t) \in C^{\infty}(\Omega \times [0, T+1])$, such that $\phi(x, T+1) = 0$; and for all mean-free, space-periodic vector-valued test functions ${\Phi}(x, t) \in [C^{\infty}(\Omega \times [0, T+1])]^2$ such that $\nabla \cdot {\Phi}(\cdot, t) = 0$, ${\Phi}(\cdot, T+1) = 0$.
				
				Note that $(u_N(\cdot-T), \theta_N(\cdot-T))$ satisfy \ref{touyingxitong}:
				\[
				\left\{
				\begin{aligned}
					&-\int_0^{T+1} \langle u_N(s-T), \Phi'(s) \rangle ds + \int_0^{T+1} \langle B(u_N(s-T), u_N(s-T)), \Phi(s) \rangle ds - \langle u_{N}(-T), \Phi(\cdot, 0) \rangle \\
					&\qquad = -\nu\int_0^{T+1}\langle \nabla u_N(s-T) , \nabla \Phi(s) \rangle ds + \int_0^{T+1} \langle \theta_N(s-T) e_2, \Phi(s) \rangle ds,  \\
					&-\int_0^{T+1}\langle \theta_N(s-T),\phi'(s) \rangle ds + \int_0^{T+1} \langle u_N \cdot \nabla \theta_N,P_N \phi(s) \rangle ds = \langle \theta_{N}(-T),\phi(\cdot, 0) \rangle.
				\end{aligned}
				\right.
				\]
				Passing to the limits as $N\to \infty$, the convergence analysis about $u$ is proceeds analogously to Step 4 of Theorem \ref{touyingcunzaiweiyi}, and the linear terms is obvious since from \eqref{Nyoujie} and \eqref{partial un} we have
				\begin{equation}\label{week1 N  star}
					(u_N(\cdot-T), \theta_N(\cdot-T)) \stackrel{*}{\rightharpoonup} (u_{\infty}(\cdot-T), \theta_{\infty}(\cdot-T)) \text{ in } L^\infty(0, T+1; \mathcal{H}) \times L^\infty(0, T+1; L^2).
				\end{equation}
				
				It remains to establish the convergence of the remaining non-linear terms of the $\theta$.
				To this end, note that from \eqref{Nyoujie} and \eqref{partial un} we also have
				\begin{equation}\label{qiangshoulian}
					(u_N(\cdot-T), \theta_N(\cdot-T)) \to (u_{\infty}(\cdot-T), \theta_{\infty}(\cdot-T)) \text{ in } L^2(0, T+1; \mathcal{H}) \times L^2(0, T+1; H^{-1}).
				\end{equation}
				Then combine  \eqref{Nyoujie}, \eqref{partial un} and \eqref{qiangshoulian}, we have

				\[
				\begin{aligned}
					&\Big|\int_0^{T+1} \langle u_N \cdot \nabla \theta_N,P_N \phi \rangle ds-\int_0^{T+1} \langle u_{\infty}\cdot \nabla \theta_{\infty},\phi\rangle ds\Big|\\
					&=\Big|\int_0^{T+1} \langle\theta_N, u_N \cdot \nabla (P_N\phi)\rangle ds - \int_0^{T+1} \langle\theta_{\infty}, u_{\infty}\cdot \nabla \phi\rangle ds \Big|\\
					&=\Big| \int_0^{T+1} \langle\theta_N - \theta_{\infty}, u_N \cdot \nabla (P_N\phi)\rangle ds
					+ \int_0^{T+1} \langle\theta_{\infty}, (u_N - u_{\infty}) \cdot \nabla (P_N\phi)\rangle ds \\
					&\quad + \int_0^{T+1} \langle\theta_{\infty}, u_{\infty}\cdot \nabla (P_N\phi - \phi)\rangle ds\Big|\\
					&	\leq \|\theta_N - \theta_{\infty}\|_{L^2(0,{T+1};H^{-1})}\|u_N\|_{L^2(0,{T+1};\mathcal{H})} \|\Delta (P_N\phi)\|_{L^{\infty}(0,{T+1};L^\infty(\Omega))} \\
					&\quad+ \|\theta_N - \theta_{\infty}\|_{L^2(0,{T+1};H^{-1})}\|u_N\|_{L^2(0,{T+1};V)} \|\nabla (P_N\phi)\|_{L^{\infty}(0,{T+1};L^\infty(\Omega))} \\
					&\quad + \|\theta_{\infty}\|_{L^\infty(0,{T+1};L^2)} \|u_N - u_{\infty}\|_{L^2(0,{T+1};\mathcal{H})} \|\nabla (P_N\phi)\|_{L^2(0,{T+1};L^\infty(\Omega))}\\
					&\quad + \|\theta_{\infty}\|_{L^\infty(0,{T+1};L^2)} \|u_{\infty}\|_{L^\infty(0,{T+1};\mathcal{H})} \int_0^{T+1} \| \nabla (P_N\phi - \phi)\|_{L^\infty(\Omega)} ds \\
					&\leq \|\theta_N - \theta_{\infty}\|_{L^2(0,{T+1};H^{-1})}\|u_N\|_{L^2(0,{T+1};\mathcal{H})} \| \phi\|_{L^{\infty}(0,{T+1};H^4(\Omega))}  \\
					&\quad +\|\theta_N - \theta_{\infty}\|_{L^2(0,{T+1};H^{-1})}\|u_N\|_{L^2(0,{T+1};V)} \| \phi\|_{L^{\infty}(0,{T+1};H^3(\Omega))}  \\
					&\quad + \|\theta_{\infty}\|_{L^\infty(0,{T+1};L^2)} \|u_N - u_{\infty}\|_{L^2(0,{T+1};\mathcal{H})} \| \phi\|_{L^2(0,{T+1};H^3(\Omega))}\\
					&\quad + \|\theta_{\infty}\|_{L^\infty(0,{T+1};L^2)} \|u_{\infty}\|_{L^\infty(0,{T+1};\mathcal{H})} \int_0^{T+1} \| \nabla (P_N\phi - \phi)\|_{L^\infty(\Omega)} ds \\
					&\to 0 \qquad \text{as $N \to \infty$.}
				\end{aligned}
				\]
				
				Consequently,
				\[	
				\left\{
				\begin{aligned}
					&-\int_0^{T+1} \langle u_{\infty}(s-T), \Phi'(s) \rangle ds + \int_0^{T+1} \langle B(u_{\infty}(s-T), u_{\infty}(s-T)), \Phi \rangle ds - \langle u_0^{(1)}, \Phi(\cdot, 0) \rangle \\
					&\qquad= -\nu\int_0^{T+1} \langle \nabla u_{\infty}(s-T) , \nabla \Phi(s) \rangle ds + \int_0^{T+1} \langle \theta_{\infty}(s-T) e_2, \Phi(s) \rangle ds, \\
					&-\int_0^{T+1} \langle \theta_{\infty}(s-T), \phi'(s) \rangle ds + \int_0^{T+1} \langle u_{\infty}(s-T) \cdot \nabla \theta_{\infty}(s-T), \phi(s) \rangle ds = \langle \theta_0^{(1)},  \phi(\cdot, 0) \rangle,
				\end{aligned}
				\right.
				\]
				since \((u_N(-T), \theta_N(-T)) {\rightharpoonup}(u^{(1)}_0,\theta^{(1)}_{0})\) in $\mathcal{H}\times L^2$. 
				Combining \eqref{Nyoujie} and \eqref{partial un} with  Lions-Magenes lemma, we have $u_{\infty}(\cdot-T)\in C(0,T+1;\mathcal{H})$. Thanks to \cite{DiPernaLions1989}, and note that  $\theta_{\infty}(\cdot-T)$ is transported by the divergence-free vector-field $u_{\infty}(\cdot-T) \in L^2(0,T+1;V)$, we can obtain in addition that $\theta_{\infty}(\cdot-T) \in C(0,T+1;L^2)$. This implies that
				\[
				(u_{\infty}(t-T),\theta_{\infty}(t-T))=S(t)(u^{(1)}_0,\theta^{(1)}_0),\qquad t\in [0, T+1],
				\]
				combining \eqref{qiang} with $(u_{N0},\theta_{N0})\stackrel{wk}{\rightharpoonup}(u_0,\theta_{0})$ in $\mathcal{H}\times L^2$, by the uniqueness of limits, we have
				\[
				(u_{\infty}(0),\theta_{\infty}(0))=(u_0,\theta_0).
				\]
				Therefore, we have \[S(T)(u^{(1)}_0,\theta^{(1)}_0)=(u_0,\theta_0).\]

				\textbf{Step 2: Construct inductions.}
				
				Proceeding recursively Step 1 for $(u_0^{(1)},\theta_0^{(1)})$, we can take a subsequence of the trajectory sequence \eqref{xing}, that is:
				$$\{(u_{N_j}(t-2T),\theta_{N_j}(t-2T)): t\in [0,T+1]\}$$ such that $(u_{N_j}(-T),\theta_{N_j}(-T))\rightharpoonup (u_0^{(1)},\theta_0^{(1)})$ and
				satisfies
				\begin{equation}\label{qiang2}
					(u_{N_j}(\cdot-2T), \theta_{N_j}(\cdot-2T)) \to (u_{\infty}(\cdot-2T), \theta_{\infty}(\cdot-2T)) \text{ in } C(0, T+1; V^{'})\times C(0, T+1; H^{-1}),
				\end{equation}
				and
				\[	(u_{N_j}(-2T), \theta_{N_j}(-2T)) {\rightharpoonup}(u^{(2)}_0,\theta^{(2)}_{0})\in \mathcal{B}_r\quad \text{ in $\mathcal{H}\times L^2$}.\]

				Then, from the process of Step 1, we can get
				\[
				(u_{\infty}(t-2T),\theta_{\infty}(t-2T))=S(t)(u^{(2)}_0,\theta^{(2)}_0),\qquad t\in [0, T+1],
				\]
				and
				\[
				S(T)(u^{(2)}_0,\theta^{(2)}_0)=(u_0^{(1)},\theta_0^{(1)}).
				\]
				
				\textbf{Step 3: Obtain the bounded complete trajectory passing through $(u_0,\theta_0)$. }
				
				Repeating Step 2 iteratively, for every $(u^{(k)}_0,\theta^{(k)}_0), k=0,1,2,\cdots$, (write $(u_0,\theta_0)$ as $(u^{(0)}_0,\theta^{(0)}_0)$), there is a $(u^{(k+1)}_0,\theta^{(k+1)}_0)\in \mathcal{B}_r$, such that 
				\[
				S(T)(u^{(k+1)}_0,\theta^{(k+1)}_0)=(u_0^{(k)},\theta_0^{(k)}).
				\]
				
				Then from the uniqueness of solution, we can define a complete trajectory $(u_{\infty}(t),\theta_{\infty}(t))$ as follows: $k=1,2,\cdots$,
				\[
				(u_{\infty}(t),\theta_{\infty}(t))=
				\left\{
				\begin{aligned}
					&	S(t+kT)(u^{(k)}_0,\theta^{(k)}_0),\quad t\in[-kT,-{(k-1)}T],\\
					&S(t)(u_0,\theta_0),\quad t\geq 0.
				\end{aligned}
				\right.
				\]

				That is, we obtain a bounded complete trajectory $\{(u_{\infty}(t),\theta_{\infty}(t)) :t\in\mathbb{R}\}\subseteq \mathcal{B}_r$ with
				$(u_{\infty}(0),\theta_{\infty}(0))=(u_0,\theta_{0})$, and it is continuous with respect to $t\in \mathbb{R}$ in $\mathcal{H}\times L^2$. Consequently, from Definition \ref{xiyinzidingyi}, it is obvious that $(u_0,\theta_{0})\in \mathcal{A}_r$.
			\end{proof}

			Now, we are ready to prove our main result, Theorem \ref{main-thm}.
			\begin{proof}[Proof of Theorem \ref{main-thm}]
				For every $\theta\in L^2$, we can find  $\theta_{N}\in  X_N$ ($N=1,2,\cdots$) such that $\theta_{N}\to \theta$, and without loss of generality, we may assume that $\|\theta_N\|_2= \|\theta\|_2=r$. Then, from Theorem \ref{touyingdingli}, there exists $u_N$ such that $(u_N,\theta_{N})\in \mathcal{A}_{Nr}$. For the sequence $\{(u_N,\theta_{N})\}_{N=1}^{\infty}$, there is subsequence $\{(u_{N_k},\theta_{N_k})\}_{k=1}^{\infty}$ such that  $(u_{N_k},\theta_{N_k})\rightharpoonup (u^{'},\theta^{'})$ in $\mathcal{H}\times L^2$. Which, thanks to Theorem \ref{bijinyinli}, implies $(u^{'},\theta^{'})\in \mathcal{A}_r$. By the uniqueness of the limit, we have $\theta^{'}=\theta$; that is, $P_{\theta} (u^{'},\theta^{'})=\theta$ and $(u^{'},\theta)\in\mathcal{A}_r$.
			\end{proof}
	
			\begin{Corollary}[Weak upper semicontinuity]
				For every $r\ge 0$, we have
				\[
				\lim_{N \to \infty}d_w(\mathcal{A}_{Nr}, \mathcal{A}_r) = 0,
				\]
				where $d_w$ denotes the metric Hausdorff semi-distance induced by the weak topology of $\mathcal{H}\times L^2$ on bounded set.
			\end{Corollary}
			\begin{proof}
				Assume the conclusion does not hold. Then there exist $\epsilon_0 > 0$ and a subsequence $\{N_k\}$ such that
				\[
				d_w(\mathcal{A}_{N_k r}, \mathcal{A}_r) \geq \epsilon_0, \quad k=1,2,\cdots.
				\]
				
				By the definition of the semi-distance,
				$
				d_w(\mathcal{A}_{N_k r}, \mathcal{A}_r) = \sup_{x \in \mathcal{A}_{N_k r}} \inf_{y \in \mathcal{A}_r} d_w(x,y)$,
				we know that, for each $k$, there exists $x_{N_k} \in \mathcal{A}_{N_k r}$ satisfying
				
				\begin{equation}\label{maodun}
					\inf_{y \in \mathcal{A}_r} d_w(x_{N_k}, y) \geq \frac{\epsilon_0}{2}.
				\end{equation}
				
				Since $\mathcal{A}_{Nr}\subseteq \mathcal{B}_r$, the sequence $\{x_{N_k}\}$ has a weakly convergent subsequence $\{x_{N_{k_j}}\}$ such that $	x_{N_{k_j}}\rightharpoonup x$. From Theorem \ref{bijinyinli}, we know that $x \in \mathcal{A}_r$, which leads to a contradiction with \eqref{maodun}.
			\end{proof}
			
			\begin{Remark}
				Generally, we can not have $\mathcal{A}_{Nr}\subseteq \mathcal{A}_r$ as $r>0$; $(0,0)\in \mathcal{A}_{Nr}\cap \mathcal{A}_r$ and it may be have some other fixed points $(0,\theta^{v})$, but whether such points really exists, or whether it contains some nontrivial and continuously evolving orbits in their intersection is unknown.
			\end{Remark}
			
			\begin{Remark}
				About systems \eqref{1.1}, it would interesting to know whether or not we can also obtain a similar result corresponding to Remark \ref{more interesting}? that is, whether or not we can obtain the following identity,
				\[
				P_{\theta} \omega^{st}(\widetilde{\mathcal{B}}_{r}^{u_0})=\{\theta \in L^2~ |~\|\theta\|_2=r\}~?
				\]
				where $\widetilde{\mathcal{B}}_{r}^{u_0}:=\{(u_0,\theta):\theta \in  L^2, \|\theta\|_{2}= r\} $ for each fixed $u_0\in \mathcal{H}$ with $\|u_0\|_{\mathcal{H}}\le R(r)$.
			\end{Remark}
			
			\section{Dimension of $P_u\mathcal{A}_r$}\label{s5}
			
			We first recall the definition of Hausdorff and fractal dimensions, see \cite{bab1992,Temam1997}.
			
			\begin{Definition}
				If $\overline{K}$ is compact in $\mathcal{H}$, the fractal dimension of $K$, $d_f(K)$, is defined by
				\[
				d_F(K) = \limsup_{\epsilon \to 0} \frac{\ln N_\epsilon(K,\mathcal{H})}{\ln(1/\epsilon)},
				\]
				where $N_\epsilon(K,\mathcal{H})$ is the least number of closed balls of a fixed radius $\epsilon$ needed to cover $K$.
			\end{Definition}
			
			\begin{Definition}
				The Hausdorff dimension of a compact set $K$, $d_H(K)$, is defined by
				\[
				d_H(K) = \inf_{d>0} \bigl\{d: \mathcal{H}^d(K) = 0\bigr\},
				\]
				where $ \mathcal{H}^d(K)$ is the $d$-dimensional Hausdorff measure.
			\end{Definition}

			We start by reviewing the following Horizontal solutions of (\ref{1.1}) introduced in \cite{Biswas2017}, its steady state plays an important role in our discussion. 
			
			\textbf{Horizontal solutions:} Let $L^2_{0,H}$ denote the zero-mean $L^2$-functions on $\Omega$ that depend only on the horizontal variable $x_1$. Let $a^H,\theta^H \in L^2_{0,H}$. Set $u_1^H = 0$, $p^H = 0$, and let $u_2^H = u_2^H(x_1,t)$ be the unique solution of the following forced diffusion problem:
			\begin{equation}
				\begin{cases}
					\partial_t u_2^H - \nu \partial_{x_1}^2 u_2^H = \theta^H,\\
					u_2^H(x_1,0) = a^H(x_1),
				\end{cases}
			\end{equation}
			with periodic boundary condition on $[0,L]$ and the mean-free condition. It is easy to check that $(\mathbf{u}^H,\theta^H,p^H) := \bigl((u_1^H,u_2^H),\theta^H,0\bigr)$ is a solution to \eqref{1.1} with initial data
			\begin{equation}
				u_1(\mathbf{x},0) = 0,\quad u_2(\mathbf{x},0) = a^H(x_1) \text{ and } \theta(\mathbf{x},0) = \theta^H.
			\end{equation}
			The horizontal solution $(\mathbf{u}^H,\theta^H) $ defined above converges, as $t\to\infty$, to the steady state
			$u_1 =0,\theta=\theta^{H}$, and $u_2$ satisfies
			\begin{equation}\label{manzudeshizi}
				\nu \frac{d^2}{dx_1^2} u_2(x_1) = - \theta^H(x_1).
			\end{equation}
			
			By the Poincare inequality,
			\[
			\|(0,u_2)\|_{\mathcal{H}}
			\leq \frac{\|\theta^{H}\|_{2}}{\nu\lambda_1}.
			\]
			Thus, whenever $\|\theta^H\|_{2}\le r$,
			\[
			\|(0,u_2)\|_{\mathcal{H}}
			\le \frac{r}{\nu\lambda_1}
			=\frac{R(r)}2.
			\]
			Consequently, $((0,u_2),\theta^H)\in \mathcal{B}_r$.  Since it is a fixed point of the semigroup, it belongs to $\omega^{wk}(\mathcal{B}_r)=\mathcal{A}_r$.  Therefore,
			\begin{equation}
				\Big\{
				\Big(\big(0,\frac{(-\partial_{x_1}^2)^{-1}\theta^H}{\nu}\big),~\theta^H\Big)\in \mathcal{H}\times L^2:~\theta^H\in L^2_{0,H},~ \|\theta^H\|_{2}\leq r
				\Big\}
				\subset \mathcal{A}_r,
			\end{equation}
			consequently,
			\begin{equation}\label{tuoqiu}
				\mathcal{E}_r^H
				:=\Big\{
				\big(0,\frac{(-\partial_{x_1}^2)^{-1}\theta^H}{\nu}\big)\in \mathcal{H}:~\theta^H\in L^2_{0,H},~ \|\theta^H\|_{2}\leq r
				\Big\}
				\subset P_u \mathcal{A}_r.
			\end{equation}
			Obviously, 
			\begin{equation}
				\text{$\mathcal{E}_r^H$ is increasing w.r.t. $r$, and $\mathcal{E}_r^H=r\mathcal{E}_1^H$}.
			\end{equation}

			We now choose an orthonormal basis of $\{e_n\}_{n=1}^{\infty}\subset L^2_{0,H}$ satisfying (note that $\kappa_n\sim n^2\kappa_1$ in $L^2_{0,H}$)
			\[
			-\partial_{x_1}^2e_n=\kappa_n e_n,~~n=1,2,\cdots.
			\]
Taking
			\[
			\theta^H=\sum_{n=1}^{\infty}a_ne_n~~\text{with}~~
			\sum_{n=1}^{\infty}|a_n|^2\le r^2,
			\]
			by \eqref{manzudeshizi} and \eqref{tuoqiu} we obtain that
			\begin{equation}\label{5.5}
				\Big(0,\sum_{n=1}^{\infty}
				\frac{a_ne_n}{\nu \kappa_n}\Big)\in\mathcal{E}_r^H.
			\end{equation}
			
			Hence, $\mathcal{E}_r^H$ is a infinite-dimensional Hilbert ellipsoid in $\mathcal{H}$, with semi-axes
			\begin{equation}\label{xiajiang}
				\alpha_n
				=\frac{r}{\nu\kappa_n}, \quad\kappa_n\sim n^2\kappa_1, \qquad n=1,2,\cdots.
			\end{equation}
			
			It is easy to see that $\mathcal{E}_r^H$ is compact in $\mathcal{H}$, convex, centrally symmetric, and has infinite-dimensional linear span. This also shows that viscosity attenuates the $n$th velocity mode by a factor of order $n^{-2}$, but it does not eliminate any of these directions. And, it is obviously that $\mathcal{E}_r^H$ has infinite Hausdorff and fractal dimensions in $\mathcal{H}$, which implies immediately the following results about the dimension estimates of $P_u\mathcal{A}_r$.
			
			\begin{Theorem}\label{thm5.1}
				For every fixed $r \ge 0$, the set $P_u \mathcal{A}_r$ is compact in the strong topology of $\mathcal{H}$. Moreover,
				
				if $r = 0$, then $P_u \mathcal{A}_0 = \{0\}$ and
				\[
				d_H (P_u \mathcal{A}_0) = d_F( P_u \mathcal{A}_0) = 0;
				\]
				
				if $r > 0$, then
				\[
				d_H (P_u \mathcal{A}_r) =  \infty,\quad d_F (P_u \mathcal{A}_r)  =  \infty.
				\]
			\end{Theorem}
Theorem \ref{thm5.1} seems has a special interesting: for the normal global attarctors (e.g., the attractors in \cite{bab1992, hale1988,Temam1997}), especially for the global attractor of 2D Navier-Stokes equations on bounded domain, its fractal dimension usually is finite; however, for system \eqref{1.1}, despite as reported in \cite{sun2019}, for each fixed $r>0$, $\mathcal{A}_r$ can attract the vilocity component $u$  even in the strong topology of $V$, the projecton $P_u\mathcal{A}_r$ is still infinite dimensional.\\

To conclude this paper, we would like to mention the following observations/disscusions.
			\begin{itemize}
				\item[-] 	From the definition of $H_r$ (see Lemma \ref{HS}), for each $r>0$, we can deduce that
				\[
				\partial\mathcal{E}_r^H:=\Big\{
				\big(0,\frac{(-\partial_{x_1}^2)^{-1}\theta^H}{\nu}\big):~\theta^H\in L^2_{0,H},~ \|\theta^H\|_{2}= r
				\Big\}\subset P_u H_r,
				\]
				where $\partial\mathcal{E}_r^H$ is the ellipsoidal surface of $\mathcal{E}_r^H$, and it also has infinite Hausdorff and fractal dimensions, and therefore implies that 
				\[
				d_H (P_u H_r) =  \infty,\quad   d_F (P_u H_r) = \infty~~\text{for each}~r>0.
				\]
				
				\item[-] A rough observation to understand the infinite dimension property above may be can through the tangent cone: if we set
				\[
				\mathcal{H}_{ho}:=\{(0,\varphi(x_1))\in\mathcal{H}:\varphi \ \text{depends only on the horizontal variable}\ x_1\},
				\]
				then the closure of the tangent cone of $P_u \mathcal{A}_r$ at the origin contains
				$\mathcal{H}_{ho}$.  Indeed, for every smooth horizontal shear
				$w=(0,\varphi(x_1))$, sufficiently small multiples $tw$ belong to
				$\mathcal{E}_r^H$.  Thus the infinite-dimensionality is present locally at the
				origin and is not merely caused by widely separated pieces of the attractor. The projection also has large fibers: for every zero-mean function
				$\theta^V=\theta^V(x_2)$ with $\|\theta^V\|_{2}\le r$, the pair
				$(0,\theta^V)$ is a vertical hydrostatic steady state, hence infinitely many distinct temperatures project to the same velocity $u=0$.
				
				\item[-] Although Theorem \ref{thm5.1} shows that the standard fractal dimension is infinite, the decay
				$\alpha_n\sim n^{-2}$ in \eqref{xiajiang} gives some substantial finite-resolution information: for the ellipsoid $\mathcal{E}_r^H$, its semi-axes are $\alpha_n=C_r/\kappa_n$, thus there exist absolute constants $c,C>0$ such that, for sufficiently small
				$\epsilon>0$,
				\begin{equation}
					c\big(\frac{C_r}{\epsilon}\big)^{1/2}
					\le
					\ln N_\epsilon(\mathcal{E}_r^H;\mathcal{H})
					\le
					C\big(\frac{C_r}{\epsilon}\big)^{1/2},
				\end{equation}
				where $C_r:=\frac{r}{\nu}$.
				This yields directly a lower bound of $\ln N_\epsilon(P_u \mathcal{A}_r;\mathcal{H})$, together with the Weyl's law for the Stokes operator on the two-dimensional torus, $\lambda_n\sim n$, gives a rough upper bound, that is
				\begin{equation}
					c'_r\epsilon^{-1/2}
					\lesssim
					\ln N_\epsilon(P_u \mathcal{A}_r;\mathcal{H})
					\lesssim
					C_r'\epsilon^{-2};
				\end{equation}
				and note that this lower bound alone also implies that
				\[
				\frac{\ln N_\epsilon(P_u \mathcal{A}_r;\mathcal{H})}{\ln(1/\epsilon)}\to\infty~~\text{as}~~\epsilon \to 0.
				\]
				
				\item[-] In \cite{wu2018, li2020, tao2020} and their related papers, the authors have investigated the long time behavior of the two-dimensional Boussinesq equations without thermal/buoyancy diffusion, and gave some interesting and comprehensive results about its dynamics, especially the dynamics near the hydrostatic equilibrium. Therefore, a natural question is whether it is possible to further decompose the weak $\sigma$-attractor based on their results, to gain some new classification or further understand its structure? 
				
			\end{itemize}
			
			\vspace{0.3 cm}

\noindent{\bf Acknowledgements} 
This work was supported by the National Natural Science Foundation of China (Grant No. 12271227).
			
			\section*{Declarations}
			
			\noindent{\bf Data Availability} Since this work is of abstract theoretical nature, no data sets are generated or analyzed.\\
			


\begin{thebibliography}{11}
				
				\bibitem{bab1992}
				A.~V.~Babin and M.~I.~Vishik.
				\emph{Attractors of evolution equations}.
				Translated and revised from the 1989 Russian original by Babin. Studies in Mathematics and its Applications, 25. North-Holland Publishing Co., Amsterdam, 1992. x+532 pp. 
				
				
				\bibitem{Carvalho2024}
				J.~Bana\'skiewicz, A.~N.~Carvalho, J.~Garcia-Fuentes, and P.~Kalita.
				Autonomous and non-autonomous unbounded attractors in evolutionary problems.
				\emph{J. Dynam. Differential Equations} \textbf{36} (2024), 3481--3534.
				
				
				
				\bibitem{Biswas2017}
				A.~Biswas, C.~Foias, and A.~Larios.
				On the attractor for the semi-dissipative Boussinesq equations.
				\emph{Ann. Inst. H. Poincaré C Anal. Non Linéaire} \textbf{34} (2017), no.~2, 381--405.
				
				
				
				
				%
				
				
				
				
				
				
				
				
				
				
				
				
				\bibitem{carvalho2013}
				A.~N.~Carvalho, J.~A.~Langa, and J.~C.~Robinson.
				\emph{Attractors for infinite-dimensional non-autonomous dynamical systems}.
				Applied Mathematical Sciences, 182. Springer, New York, 2013. xxxvi+409 pp. 
				
				
				
				
				\bibitem{dong2006}
				D.~Chae.
				Global regularity for the 2D Boussinesq equations with partial viscosity terms.
				\emph{Adv. Math.} \textbf{203} (2006), no.~2, 497--513.
				
				\bibitem{chepyzhov1992}
				V.~V.~Chepyzhov,  A.~Yu.~Goritskiĭ.
				Unbounded attractors of evolution equations. Properties of global attractors of partial differential equations.
				\emph{Adv. Soviet Math.} \textbf{10} (1992), Amer. Math. Soc., Providence, RI, 85–128.
				
				
				
				
				
				\bibitem{Rapha2008}
				R.~Danchin and M.~Paicu.
				The Leray and Fujita-Kato theorems for the Boussinesq system with partial viscosity.	\emph{Bull. Soc. Math. France} \textbf{136} (2008), no.~2, 261--309.
				
				
				
				
				
				
				
				
				\bibitem{DiPernaLions1989}
				R.~J. DiPerna and P.-L. Lions.
				Ordinary differential equations, transport theory and Sobolev spaces.
				\emph{Invent. Math.} \textbf{98} (1989), 511--547.
				
				\bibitem{wu2018}
				C.~R.~Doering, J.~ Wu, K.~Zhao, and X.~ Zheng. Long time behavior of the two-dimensional Boussinesq equations without buoyancy diffusion. Phys. D 376/377 (2018), 144–159.
			
				%
				%
			
				\bibitem{hale1988}
				J.~K.~Hale.
				\emph{Asymptotic behavior of dissipative systems}.
				Mathematical Surveys and Monographs, 25. American Mathematical Society, Providence, RI, 1988. x+198 pp. 
			
				
				%
				%
				%
				
				
				
				\bibitem{sun2019}
				J.~He and C.~Sun.
				The weak sigma-attractor for the semi-dissipative 2D Boussinesq system. 
				\emph{Proc. Amer. Math. Soc.} \textbf{148} (2020), no.~3, 1219--1231.
				
				
				
				
				%
				
		
				
				\bibitem{ladyzhenskaya1991}
				O.~Ladyzhenskaya.
				\emph{Attractors for semigroups and evolution equations}.
				Lezioni Lincee. [Lincei Lectures] Cambridge University Press, Cambridge, 1991. xii+73 pp. 
		
				\bibitem{adam2013}
				A.~Larios, E.~Lunasin, and E.~S.~Titi.
				Global well-posedness for the 2D Boussinesq system with anisotropic viscosity and without heat diffusion.
				\emph{J. Differential Equations} \textbf{255} (2013), no.~9, 2636--2654.
				
				
				\bibitem{li2020}
				B.~Li, F.~Wang, and K.~Zhao.
				Large time dynamics of 2D semi-dissipative Boussinesq equations.\emph{ Nonlinearity} \textbf{33} (2020), no.~5, 2481–2501. 
				
			
				
				\bibitem{29}
				A.~Majda and A.~Bertozzi.
				\emph{Vorticity and incompressible flow}.
				Cambridge Texts in Applied Mathematics, 27. Cambridge University Press, Cambridge, 2002. xii+545 pp.
			
				
				\bibitem{oregan2006}
				D.~O'Regan, Y.~J.~Cho, and Y.~Q.~Chen.
				\emph{
					Topological degree theory and applications}.
				Series in Mathematical Analysis and Applications, 10. Chapman \& Hall/CRC, Boca Raton, FL, 2006. iv+221 pp.
				
				\bibitem{31}
				J.~Pedlosky.
				\emph{Geophysical Fluid Dynamics}.
				Springer, New York, 1987. XIV, 710 p.
		
				\bibitem{xu2025}
				A.~Stefanov, J.~Wu, X.~Xu, and Z.~Ye.
				Global regularity results of the 2D fractional Boussinesq equations.
				\emph{Math. Ann.} \textbf{391} (2025), no. 4, 5965–6012. 
				
				\bibitem{tao2020}
				L.~Tao, J.~ Wu, K.~ Zhao, and X.~ Zheng.
				Stability near hydrostatic equilibrium to the 2D Boussinesq equations without thermal diffusion.\emph{ Arch. Ration. Mech. Anal.}\textbf{ 237} (2020), no. 2, 585–630.
				
				
				\bibitem{Temam1997}
				R.~Temam.
				\emph{Infinite-dimensional dynamical systems in mechanics and physics}.
				Applied Mathematical Sciences, 68. Springer-Verlag, New York, 1997. xxii+648 pp.
				
			\end{thebibliography}
	\end{document}